\documentclass[11pt, a4paper]{article}
\usepackage[utf8]{inputenc}
\usepackage[margin=1.1in]{geometry}
\usepackage{graphicx}
\usepackage{amsmath, amssymb}
\usepackage{amsthm}
\usepackage{booktabs}
\usepackage[numbers,sort&compress]{natbib}
\usepackage{hyperref}
\usepackage{titlesec}
\usepackage{enumitem}
\usepackage{xcolor}
\usepackage{bm,mathrsfs}

\newtheorem{theorem}{Theorem}[section]

\newtheorem{lemma}[theorem]{Lemma}

\newtheorem{remark}[theorem]{Remark}
\newtheorem{definition}[theorem]{Definition}

\def\u{{\bf u}}
\def\vv{{\bf v}}
\def\w{{\bf w}}
\def\n{{\bf n}}
\def\f{{\bf f}}
\def\h{{\bf h}}
\def\e{{\bm{\eta}}}
\def\si{{\bm{\sigma}}}
\def\H{{\bf H}}

\def\x{{\bf x}}

\def\H{{\bf H}}

\def\L{{\bf L}}
\def\d{{\mathrm d}}
\def\X{{\mathrm X}}

\def\ue{{\mu_1}}
\def\ur{{\mu_2}}

\newcommand{\D}{\mathbb{D}}

\newcommand{\bbeta}{\boldsymbol{\beta}}

\usepackage[textsize=tiny,color=blue!50!white]{todonotes}

\titleformat{\section}{\large\bfseries}{\thesection}{1em}{}
\titleformat{\subsection}{\bfseries}{\thesubsection}{1em}{}
\setlist[itemize]{noitemsep, topsep=0pt}

\title{\Large\textbf{Weighted $\H^2$ regularity of fluid-structure interaction\\ when the interface intersects the boundary}}
\author{
    Zhaonan Dong\textsuperscript{1}, 
    \,\,Tiantian Huang\textsuperscript{2}, 
    \,\,Buyang Li\textsuperscript{3}
}

\date{}

\begin{document}
\maketitle
\begin{center}
\small
$^{1}$ Inria Paris, 48 rue Barrault, 75647 Paris, France\\
$^{1}$ CERMICS, École nationale des ponts et chaussées,
IP Paris, 77455 Marne-la-Vall\'ee, France\\
$^{2}$Department of Applied Mathematics, The Hong Kong Polytechnic University,
Hong Kong\\  
$^{3}$Department of Applied Mathematics, The Hong Kong Polytechnic University,
Hong Kong\\
\end{center}
\begin{abstract}
This paper analyzes the regularity of solutions to a fluid-structure interaction (FSI) problem involving a Stokesian fluid and a linear elastic solid in a two-dimensional polygonal domain, seperated by an interface which intersects the boundary of the domain. The main challenges arise from the limited solution regularity caused by geometric singularities at the domain corners. To address this, we introduce tailored weighted Sobolev spaces, denoted by $\mathbf{H}^{k,l}_{\boldsymbol{\beta}}$, and a novel solution operator framework on these spaces. This framework provides the key insight for decoupling the coupled FSI system into tractable fluid and elastic solid subproblems in the regularity analysis. As our main result, we prove the existence and uniqueness of a solution  in the weighted Sobolev space $\mathbf{H}^{2,2}_{\boldsymbol{\beta}}$.
\end{abstract}

\section{Introduction}\label{sec:introduction}
Fluid-structure interaction (FSI) describes the multiphysics phenomenon wherein a fluid flow and a deformable structure exert mutual influence on each other. The coupling is bidirectional: the fluid's pressure and shear stress load the structure, inducing its deformation or motion, which subsequently modifies the fluid domain and flow field. Such interactions arise in a wide range of applications, including biomechanics, aeroelasticity, and engineering design; see, for instance, \cite{Bodnr2014FluidstructureIA,chakrabarti2002theory,Buka2019FluidstructureIB,KAMAKOTI2004535,wang2022mathematical,richter2017fluid}. As a result, FSI problems have been the subject of extensive analytical and numerical investigation.

A substantial body of literature is devoted to the well-posedness and regularity of solutions to FSI systems. Early analytical studies primarily focused on FSI models formulated on a fixed reference domain  \cite{avalos2008higher,avalos2007coupled,avalos2008uniform,avalos2009semigroup,avalos2013fluid,barbu2007existence,barbu2008smoothness,kukavica2010strong,Kukavica2009StrongST,zhang2007long}. Such models provide accurate approximations when structural displacements and deformations remain small. When the structure undergoes large displacements, however, the resulting deformation has a significant impact on the fluid dynamics, and the fluid domain becomes time-dependent and fully coupled to the structural motion. This leads to geometrically nonlinear FSI models, which have been extensively studied; see, for example, \cite{Muha2012ExistenceOA,Ignatova2012OnWF,Ignatova2014OnWA,Kukavica2011SOLUTIONSTA,Muha2013AN3,Muha2013ExistenceOA,Muha2015FluidstructureIB,Muha2015ExistenceOA,Kampschulte2022UnrestrictedDO,Muha2019ExistenceAR,Schwarzacher2020WeakStrongUF}. 
The geometric nonlinearities induced by the moving fluid domain substantially complicate the analysis of well-posedness and regularity when compared with FSI models posed on fixed domains. In the present work, we consider a two-dimensional FSI system consisting of a Stokes fluid coupled with a linearly elastic solid, posed on a fixed domain.

The coupled parabolic-hyperbolic system formed by the Stokes and Lamé equations has been extensively studied, and a substantial body of literature is devoted to establishing well-posedness and regularity properties of its solutions. Existing results may be broadly classified according to the geometric setting and the type of coupling. Throughout this discussion, $\mathbf{u}$ and $p$ denote the fluid velocity and pressure, respectively, while $\boldsymbol{\eta}$ represents the solid displacement.
    \begin{itemize}
        \item \textbf{Lipschitz domains} (where the interface may intersect the boundary).  The existence and regularity of weak solutions for the Stokes-Lam\'e systems in Lipschitz domains were studied in \cite{du2004semidiscrete}. For initial data $(\e_0,\e _1,\u_0)\in \H^1\times \H^1\times \H^1$, the existence and uniqueness of a weak solution  $(\e,\e_t ,\u)\in \H^1\times \L^2\times \L^2$ with $\u \in L^2(0,T;\H^1)$ was established in \cite[Theorem 2.2]{du2004semidiscrete}. For smoother initial data $(\e_0,\e _1,\u_0)\in \H^2\times \H^1\times \H^2$ and an associated initial pressure $p_0\in H^1$, improved regularity $(\e,\e_t ,\u)\in\H^1\times \H^1\times \L^2$ was obtained in \cite[Theorem 2.3]{du2004semidiscrete}.\medskip
        
        \item \textbf{Smooth domains} (where the interface does not intersects the boundary). For FSI systems posed on smooth domains with non-intersecting interfaces, a series of regularity results was established in \cite{avalos2008higher,avalos2007coupled,avalos2008uniform,avalos2009semigroup,avalos2013fluid}. In particular, for the Stokes-Lam\'e coupling system,  initial data $(\e_0,\e _1,\u_0)\in  \H^1\times \L^2\times \L^2$ yield a strong solution $(\e,\e_t ,\u)\in \H^1\times \L^2\times \L^2$ with $\u \in L^2(0,T;\H^1)$, while initial data $(\e_0,\e _1,\u_0)\in  \H^1\times \H^1\times \H^1$ imply $(\e,\e_t ,\u)\in \H^1\times \H^1\times \H^1$; see \cite[Theorem 4.1]{avalos2009semigroup}. For the  Navier-Stokes-wave equation coupling system, higher-regularity initial data $(\e_0,\e _1,\u_0)\in \H^2\times \H^1\times \H^1$ yield a strong solution $(\e,\e_t,\u)\in \H^2\times \H^1\times \H^1$ with $\u \in L^2(0,T;\H^2)$; see \cite[Theorem 2.1]{avalos2008higher}. The analysis in \cite{avalos2013fluid} extends these results to FSI models within the fundamental $\H^1 \times \L^2 \times \L^2$ framework.\medskip
        
        \item \textbf{Other settings.} Fluid–structure interaction models coupling the Navier-Stokes or heat equations with dynamic elasticity--typically in configurations where the interface does not intersect the boundary--were studied in  \cite{barbu2007existence,barbu2008smoothness,kukavica2010strong,Kukavica2009StrongST}. 
        In particuloar, local-in-time existence of strong solutions $(\e,\e_t,\u)\in \H^2\times \H^1\times \H^1$ for initial data $(\e_0,\e_1,\u_0)\in \H^2\times \H^1\times \H^2$ was proved in \cite{barbu2008smoothness}. The long-time existence of strong solutions $(\e,\e_t,\u)\in \H^2\times \H^1\times \H^2$ for initial data $(\e_0,\e_1,\u_0)\in \H^2\times \H^1\times \H^2$ was shown in \cite[Theorem 3]{zhang2007long}. 
    Moreover, the local-in-time existence of solutions $(\e,\e_t,\u)\in \H^{3/2+k}\times \H^{1/2+k}\times \H^1$ for initial data $(\e_0,\e_1,\u_0)\in \H^{3/2+k}\times \H^{1/2+k}\times \H^1$, with $0<k<(\sqrt{2}-1)/2$, was shown in \cite[Theorem 2.1]{kukavica2010strong}. 
 
    \end{itemize}


Despite these advances, the analysis of FSI problems in the $\H^2\times \H^1\times \H^2$ framework has largely been restricted to smooth domains, flat geometries \cite{Kukavica2009StrongST,kukavica2010strong}, or configurations with $C^2$ boundaries \cite{zhang2007long}, where the fluid-structure interface does not intersect the boundary. This geometric assumption ensures the $\H^2$-regularity of $\boldsymbol{\eta}$. This level of regularity---or, in the presence of geometric singularities, an appropriate weighted $\H^2$ regularity---holds not merely theoretical value: it is crucial for the design of finite element methods used in practical FSI computations, as it underpins optimal approximation properties and a priori error estimates, and it guides the choice of stabilization mechanisms and mesh refinement strategies (e.g., graded meshes near singular points) needed to recover accuracy and robustness. 
In particular, a significant gap remains in the analysis of regularity of solutions in nonsmooth domains where fluid-structure interfaces intersect the boundary, leading to corner or edge singularities at the intersection points. 
Consequently, further investigation of complex geometrical configurations is required, especially near corners and edges of polygonal and polyhedral domains, where solution behaviour, (weighted) $\H^2$ regularity, and their implications for reliable finite element discretisations remain poorly understood.

  In polygonal or polyhedral domains, geometric singularities such as corners and edges generally lead to a loss of classical Sobolev regularity. 
  
\medskip
\noindent
\textbf{Corner singularities and existing theory.}
Corner singularities for elliptic systems in
polygonal domains have been studied in depth over the past decades \cite{doi:10.1137/120867937, dauge2006elliptic,doi:10.1137/1.9781611972030,li2020explicit,rossle2000corner}. For single equations or uncoupled systems, such as the Stokes equations or linear elasticity, solutions generally fail to belong to $\H^2$ near corner points. Instead, they admit local asymptotic expansions determined by the local geometry and boundary conditions (see \cite{li2020explicit,rossle2000corner}): 

            \begin{equation}\label{corner_singularity:single_pde}
                \mathbf{u} = \mathbf{u}_{\text{reg}} + \sum_{j,k} (\ln r)^k r^{\lambda_j} \mathbf{U}_{j,k}(\theta)
            \end{equation}
            where:
            \begin{itemize}
                \item $r,\theta$: Polar coordinates at the corner
                \item $\lambda_j$: Singularity exponents ($\operatorname{Re}(\lambda_j) >0$), depending on the
opening angle and the imposed boundary conditions 
                \item $\mathbf{u}_{\text{reg}} \in \H^2$: Regular part
            \end{itemize}


\medskip
\noindent
While this theory provides a precise description of classical corner singularities, it is inherently limited to single-domain or uncoupled systems. In FSI problems, new singularities arise at the interface-boundary intersection points, where the fluid-structure interface meets the fixed external boundary. At such points, the governing equations, boundary conditions, and interface conditions change simultaneously. The resulting singular behavior cannot be attributed to either the fluid or the solid alone; it is induced by the coupling itself. Consequently, classical $\H^2$ theory is inapplicable, and any second-order regularity estimate must be derived for the fully coupled system.


\medskip
\noindent
    \textbf{Mathematic Model.} In this paper, we study a two-dimensional FSI system consisting of a Stokes fluid coupled with a linear elastic solid in a polygonal domain $\Omega= \text{int}(\overline{\Omega}_f \cup \overline{\Omega}_s)$. The fluid and solid occupy subdomains  $\Omega_f$ and $\Omega_s$, respectively,  separated by an interface $\Gamma = \partial \Omega_f \cap \partial \Omega_s$. The remaining external boundaries are denoted by $\Gamma_f = \partial \Omega_f \setminus \Gamma$ for the fluid and $\Gamma_s = \partial \Omega_s \setminus \Gamma$ for the solid; see Figure~\ref{fig:FSI_polygonal_domain}.

        \begin{figure}
\centering
\includegraphics[width=0.57\linewidth]{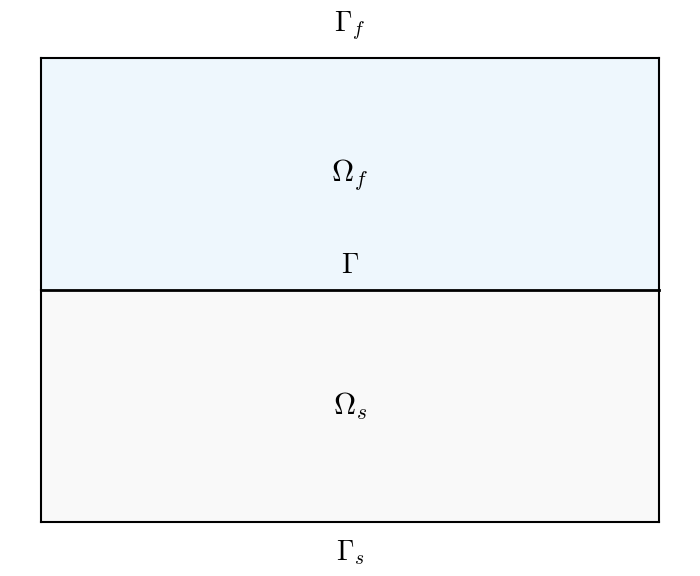}
\caption{\label{fig:FSI_polygonal_domain}Geometry of the fluid-structure interaction problem: fluid domain $\Omega_f$,  solid domain $\Omega_s$, and  interface $\Gamma$.}
\end{figure}
	
		Over a fixed time interval $(0, T]$, the coupled system is governed by the following equations:
	\begin{equation}\label{eq:stoke_elasticity_t}
		\begin{cases}
			\partial_t \u-\nabla\cdot (2\mathbb{D}(\mathbf{u}) - p \mathbb{I})=\f &\text{in}\,\,\Omega_f\times (0,T],\\
			\nabla\cdot \mathbf{u}=0 &\text{in}\,\,\Omega_f\times (0,T],\\
			\partial_{tt}\e-\nabla\cdot (2\mu_1\mathbb{D}(\e ) +\mu_2(\nabla\cdot \e ) \mathbb{I})=\mathbf{h}&\text{in} \,\,
			\Omega_s\times (0,T],
		\end{cases}
	\end{equation}
where the fluid stress tensor is $\si_f(\mathbf{u}, p) = 2 \mathbb{D}(\mathbf{u}) - p\mathbb{I}$, the solid stress tensor is $\si_s(\e) = 2\mu_1\mathbb{D}(\e ) +\mu_2(\nabla\cdot \e ) \mathbb{I}$, and $\mathbb{D}(\mathbf{v}) = \frac{1}{2}(\nabla \mathbf{v} + (\nabla \mathbf{v})^{\rm T})$ denotes the symmetric gradient. In these equations, $\mathbf{u}$ and $p$ represent the fluid velocity and pressure, while $\bm{\eta}$ is the solid displacement. The system is driven by the body force per unit mass $\mathbf{f}$ in the fluid and $\mathbf{h}$ in the solid, and the solid's elastic response is characterized by the positive Lam\'{e} parameters $\mu_1$ and $\mu_2$.
    
	The following interface conditions couple the subproblems on $\Gamma \times (0,T]$:
	\begin{equation}\label{con:interface}
		\begin{cases}
\partial_t \e  = \mathbf{u} & \text{(kinematic condition)}, \\
\si_f(\mathbf{u}, p) \mathbf{n} = \si_s(\e ) \mathbf{n} & \text{(traction condition)},
\end{cases}
	\end{equation}
	where $\mathbf{n}$ is the unit outward normal to $\Omega_f$ on $\Gamma$.

    Boundary and initial conditions complete the problem:\begin{equation}\label{con:boundary}
		\begin{cases}
			\mathbf{u}=\u_{in}&\text{on}\,\,\Gamma_{f}\times (0,T],\\
			\si_s(\e ) \mathbf{n}=\mathbf{0}&\text{on}\,\,\Gamma_s\times (0,T],
		\end{cases}
	\end{equation}
 and
	\begin{equation}\label{con:initial}
		\begin{cases}
			\u=\u_0&\text{on}\,\,\Omega_f\times \{t=0\},\\
			\e=\e_0&\text{on}\,\,\Omega_s\times \{t=0\},\\
			\partial_t\e=\e_1&\text{on}\,\,\Omega_s\times \{t=0\}.\\
		\end{cases}
	\end{equation}

    A distinctive feature of this configuration is that the interface $\Gamma$ has endpoints, where the fluid-structure
interface $\Gamma$ intersects the fixed external boundary of the domain.  At each endpoint $\x_0\in \overline{\Gamma}\cap\partial\Omega$,  the interface meets the external boundary and the type of boundary condition changes abruptly. At such points, the following conditions must be satisfied simultaneously:
    \begin{itemize}
        \item [(i)] prescribed boundary data on $\Gamma_f$ and $\Gamma_s$,
        \begin{equation*}
            \u(\x_0)=\u _{in}(\x_0),\quad \si_s(\e(\x_0) ) \mathbf{n}=\mathbf{0};
        \end{equation*}
        \item [(ii)] velocity continuity across the interface $\Gamma$,
        \begin{equation*}
            \u (\x_0)=\partial_t \e(\x_0);
        \end{equation*}
        \item [(iii)] stress continuity across $\Gamma$,
        \begin{equation*}
            \si_f(\mathbf{u}(\x_0), p(\x_0)) \mathbf{n} = \si_s(\e(\x_0) ) \mathbf{n}.
        \end{equation*}
    \end{itemize}
    These conditions are generally incompatible in the classical sense, giving rise to singularities that are not simple superpositions of fluid and solid corner singularities. The fluid and solid singularities are coupled, and the behavior of one subsystem directly influences the other.

    Despite the extensive literature on corner singularities for single PDEs, a rigorous $\H^2$-type regularity theory for FSI systems in polygonal domains---particularly in the presence
of interface-boundary corner singularities---has been lacking. In particular, it remains unclear how to obtain second-order spatial regularity estimates that are compatible with the coupled nature of the problem.

\medskip
\noindent
\textbf{Main contribution and methodological novelty.} 
The main objective of this work is to establish an $\H^{2,2}_{\bbeta}$-regularity theory  for  fluid-structure interaction (FSI) problems posed in polygonal domains with interface-boundary corner singularities. Rather than deriving explicit asymptotic expansions such as  \eqref{corner_singularity:single_pde}---which are prohibitively complex for coupled systems---we adopt a weighted Sobolev framework $\H^{k,l}_{\bbeta_i}(\Omega_i)$ ($k\ge l\ge 0$, $i=f,s$) following \cite{babuvska1988h,babuvska1988regularity,babuvska1989regularity}. Within this setting,  we identify critical weight exponents $\bbeta_i$ such that
\begin{equation}
        \|\u(t)\|_{\H^{2,2}_{\bbeta_f}(\Omega_f) }+\|\e (t)\|_{\H^{2,2}_{\bbeta_s} (\Omega_s)}\leqslant
        C.
    \end{equation}
The weighted framework controls second derivatives near geometric singularities through the weight $r(x)^{\bbeta}$, where $r(x)$ denotes the distance to a corner. The exponent $\bbeta$ compensates for the loss of regularity, yielding finite $\H^2$-type norms even when classical $\H^2$ estimates fail. In this way, weighted Sobolev spaces capture the essential singular behavior without requiring explicit asymptotic expansions--an approach particularly effective for coupled systems in non-smooth domains.

A key structural feature of the analysis is that $\H^2$-regularity for the fluid and the solid cannot be established separately. The corner singularity intrinsically couples the Stokes and elasticity equations, and second-order estimates must therefore be derived for the fully coupled system.



Our methodology is based on an operator framework that reformulates the coupled FSI system in terms of solution operators for the Stokes and elasticity subproblems. We introduce Neumann, Dirichlet, and force operators ($N_\Gamma$, $\tilde{N}_\Gamma$, $D_{\Gamma_f}$, $D_{\Gamma}$, $F_1$ and $F_2$)  which map interface data, boundary data, and source terms to the corresponding fluid and solid solutions. This construction allows the coupled solution $(\mathbf{u},p,\boldsymbol{\eta})$ to be expressed implicitly through operator compositions acting on the interface variables. Specifically, 
 \begin{equation}
            \begin{cases}
                (\mathbf{u},p) = N_\Gamma \boldsymbol{\eta} + D_{\Gamma_f}\mathbf{u}_{in} + F_1(\mathbf{f}-\partial_t\mathbf{u}), \\
                \boldsymbol{\eta} = D_\Gamma\left(\int_0^t \mathbf{u}(s)ds + \boldsymbol{\eta}_0\right) + F_2(\mathbf{h}-\partial_{tt}\boldsymbol{\eta}).
            \end{cases}
        \end{equation}
        Substituting the expression for $\boldsymbol{\eta}$ into the fluid equation yields a closed equation for  $\mathbf{u}$:
	\begin{align}\label{eq:integral_u}
    (\u,p)&=N_{\Gamma}D_\Gamma\int_{0}^{t}\u (s) \d s+N_\Gamma D_{\Gamma}\e_0+D_{\Gamma_f}\u_{in}+N_\Gamma F_2(\h-\partial_{tt}\e)+F_1(\f-\partial_t\u),\notag\\
\u&=(N_{\Gamma})_1D_\Gamma\int_{0}^{t}\u (s) \d s+(N_\Gamma)_1 D_{\Gamma}\e_0+(D_{\Gamma_f})_1\u_{in}+(N_{\Gamma})_1 F_2(\h-\partial_{tt}\e)+(F_1)_1(\f-\partial_t\u).
	\end{align}
      Introducing the auxiliary variable $\mathbf{v}(t) = \int_{0}^{t} \mathbf{u}(s) \d s$ transforms the integral equation  into the evolution equation
	\begin{align}
		\partial_t\vv&=(N_\Gamma)_1D_{\Gamma}\vv(t)+(N_\Gamma)_1D_{\Gamma}\e_0+(D_{\Gamma_f})_1\u_{in}+(N_\Gamma)_1 F_2(\h-\partial_{tt}\e)+(F_1)_1(\f-\partial_t\u).
	\end{align}
By the weighted regularity theory for Stokes and elasticity systems \cite{guo1993regularity,guo2006analytic},  the composition operator $(N_\Gamma)_1D_{\Gamma}$ is bounded and linear  on $\H^{2,2}_{\bbeta_f}(\Omega_f)$. Consequently, by  standard semigroup theory \cite[Theorem 1.2, Corollary 1.4]{pazy2012semigroups}, the operator $(N_\Gamma)_1 D_{\Gamma}$ generates a uniformly continuous  semigroup $\{T(t)\}_{t\ge0}$ on  $\mathbf{H}^{2,2}_{\bm{\beta}_f}(\Omega_f)$. 

The full solution is then reconstructed in a manner that rigorously preserves the coupling conditions. In particular, there exists $\vv \in \H^{2,2}_{\bbeta_f}(\Omega_f)$ such that
 \begin{equation}
                \begin{cases}
                    \boldsymbol{\eta} = D_\Gamma(\mathbf{v}(t) + \boldsymbol{\eta}_0) + F_2(\mathbf{h}-\partial_{tt}\boldsymbol{\eta}) ,\\
                    (\mathbf{u},p) = N_\Gamma \boldsymbol{\eta} + D_{\Gamma_f}\mathbf{u}_{in} + F_1(\mathbf{f}-\partial_t\mathbf{u}). 
                \end{cases}
            \end{equation}

            \medskip
\noindent
\textbf{Organization of the paper.} The remainder of this paper is organized as follows. Section~\ref{sec:Notations} introduces the weighted Sobolev spaces and states the main regularity results. Section~\ref{sec:proof} develops the solution-operator framework and presents the proof of the main theorem. Section~\ref{sec:conclusion} concludes the paper and outlines directions for future research.

\section{Notations and Main Results}\label{sec:Notations}
\subsection{The weighted Sobolev spaces $H^{k,l}_{\beta}$}

Throughout this paper, boldface symbols such as $\bm{X}$ are used to indicate spaces of vector-valued functions, while the plain symbol $X$ refers to the corresponding scalar space.

We use the notations and definitions for function spaces introduced in \cite{babuvska1988regularity,babuvska1989regularity,guo1993regularity, guo2006analytic,muller2017symmetric}. As pictured in Figure~\ref{fig:FSI_corner_polygonal_domain}, let $\Omega = \Omega_f \cup \Omega_s \cup \Gamma \subset \mathbb{R}^2$ be a polygon with vertices $A_i$ and open edges $\Gamma_i$ connecting $A_i$ and $A_{i+1}$, $0 \leq i \leq M \,(A_{M+1}=A_0)$. Assume there exists an index $J\in \{2,3,\cdots,M-1\}$ such that the interface is $\Gamma=A_0A_{J}$ (a line segment), the fluid boundary is $\Gamma_f := \bigcup_{j \in \{0,1,\cdots,J-1\} }\Gamma_j$, and the structure boundary is $\Gamma_s := \bigcup_{j \in \{J,J+1,\cdots,M\} } \Gamma_j$.  

Let $\{\omega_{f,i}\in (0,2\pi): i=0,1,\cdots,J\} $ denote the interior angles of $\Omega_f$ at its vertices $A_i$. Similarly, let $\{\omega_{s,i}\in (0,2\pi): i=0,1,\cdots,M+1-J\} $ denote the interior angles of $\Omega_s$ at its vertices $A_{M+1-i}$. For brevity, we set $J_f=J$, $J_s=M+1-J$. 

 \begin{figure}
\centering
\includegraphics[width=0.7\linewidth]{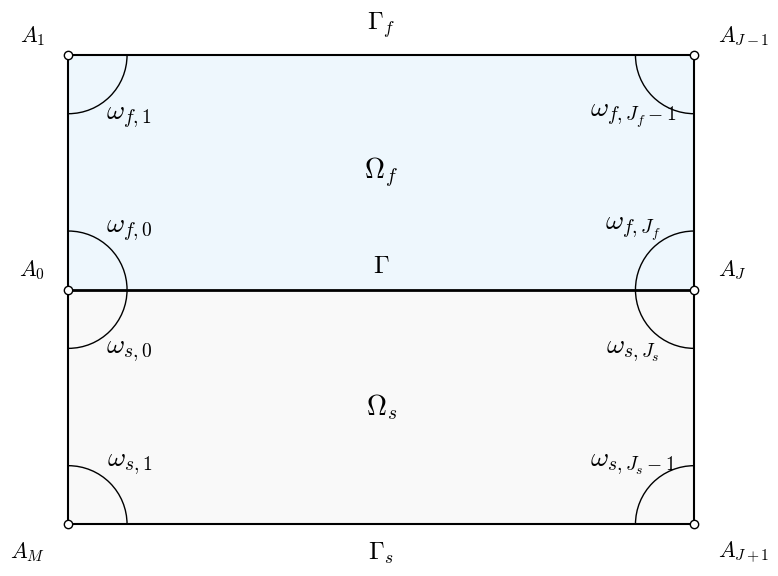}
\caption{\label{fig:FSI_corner_polygonal_domain}Fluid-structure interaction configuration in a polygonal domain with interface-boundary corners.}
\end{figure}

Furthermore, let $\bm{\beta}_f = (\beta_{f,0}, \ldots, \beta_{f,J_f})$ be an $(J_f+1)$-tuple of real numbers and $\bm{\beta}_s = (\beta_{s,0}, \ldots, \beta_{s,J_s} )$ be an $(J_s+1)$-tuple of real numbers, all satisfying $ \beta_{f,j},\beta_{s,j} \in (0,1)$. We define the weight functions
\begin{align}
    \Phi_{\bm{\beta}_f+k}(\x) &:= \prod_{i=0}^{J_f} r_i(\x)^{\beta_{f,i} + k},\\
    \Phi_{\bm{\beta}_s+k}(\x) &:= \prod_{i=0}^{J_s} r_i(\x)^{\beta_{s,i} + k},
\end{align}
where $r_i(\x) = \text{dist}(\x, A_i)$ denotes the distance from point $\x$ to the vertex $A_i$, $k$ is any integer.
\begin{remark}
    To analyze the localization properties of the weight functions $\Phi_{\bm{\beta}_f}(\x)$ and $\Phi_{\bm{\beta}_s}(\x)$, define local corner domains $D_i$ by
\begin{equation}\label{local_cone}
    D_i:=\{\x \in \Omega:|\x-A_i|<R_i\},\qquad 0 \leq i \leq M,
\end{equation}
where $0<R_i<\frac{1}{2}\min _{i\neq j}|A_i-A_j|$. This ensures the domains $D_i$ are mutually disjoint.

If $\beta_{f,i}, \beta_{s,i} \in(0,1)$, then we have
    \begin{align}
        C_{dc}^{-1}\leq &r_i(\x)^{\beta_{f_i}}\leq C_{dc},\qquad\x \in \Omega_f\setminus\cup_{j=0}^MD_j,\notag \\
        C_{dc}^{-1}\leq &r_i(\x)^{\beta_{s_i}}\leq C_{dc},\qquad\x \in \Omega_s\setminus\cup_{j=0}^MD_j,\notag \\
        C_{dc}^{-1}\Phi_{\bm{\beta}_f}(\x)\leq &r_i(\x)^{\beta_{f_i}}\leq C_{dc}\Phi_{\bm{\beta}_f}(\x),\qquad\x \in \Omega_f\cap D_i,\notag \\
        C_{dc}^{-1}\Phi_{\bm{\beta}_s}(\x)\leq &r_i(\x)^{\beta_{s_i}}\leq C_{dc}\Phi_{\bm{\beta}_s}(\x),\qquad\x \in \Omega_s\cap D_i,
    \end{align}
     for a constant $C_{dc}>0$ depending on the radii $R_i$ in \eqref{local_cone}.
\end{remark}

Let $H^k(\Omega_f)$, $H^k(\Omega_s)$  denote the standard Sobolev spaces. For integers $k, \ell$ with $k \geq \ell \geq 0$,  the weighted Sobolev spaces $H^{k,\ell}_{\bm{\beta}_f}(\Omega_f)$, $H^{k,\ell}_{\bm{\beta}_s}(\Omega_s)$  are equipped with the norms
\begin{align}
    \|{u}\|_{H^{k,\ell}_{\bm{\beta}_f}(\Omega_f)}^2& := \|{u}\|_{H^{\ell-1}(\Omega_f)}^2 + \sum_{|\alpha| \geq \ell}^k \|{\Phi_{\bm{\beta}_f+|{\alpha}| - \ell} D^\alpha u}\|^2_{L^2(\Omega_f)},\\
    \|{u}\|_{H^{k,\ell}_{\bm{\beta}_s}(\Omega_s)}^2& := \|{u}\|_{H^{\ell-1}(\Omega_s)}^2 + \sum_{|\alpha| \geq \ell}^k \|{\Phi_{\bm{\beta}_s+|{\alpha}| - \ell} D^\alpha u}\|^2_{L^2(\Omega_s)},
\end{align}
where the term $\|{u}\|_{H^{\ell-1}(\Omega_f)}$ or $\|{u}\|_{H^{\ell-1}(\Omega_s)}$ is omitted if $\ell = 0$. In the special case $k = \ell = 0$, we use the notation ${L}^2_{\bm{\beta}_f}(\Omega_f)$ and ${L}^2_{\bm{\beta}_s}(\Omega_s)$ for the spaces $H^{0,0}_{\bm{\beta}_f}(\Omega_f)$ and $H^{0,0}_{\bm{\beta}_s}(\Omega_s)$, respectively. Here, we use standard multi-index notation for higher derivatives in Cartesian coordinates:
\begin{equation}
    D^\alpha u = \frac{\partial^{|{\alpha}|} u}{\partial x_1^{\alpha_1} \partial x_2^{\alpha_2}}, \quad \alpha = (\alpha_1, \alpha_2), \quad |{\alpha}| = \alpha_1 + \alpha_2.
\end{equation}

Next,we define the trace of functions in the weighted Sobolev spaces  $H^{k,\ell}_{\bm{\beta}_i}(\Omega_i)$, $i=f,s$. Let $\gamma$ be a subset (or the entirety) of the boundary $\partial \Omega_i$. For integers $k \geq l \geq 1$, the space $H^{k-1/2,l-1/2}_{\bbeta_i}(\gamma)$ is defined as the set of all functions $\varphi$ on $\gamma$ for which there exists a lifting $\chi\in H^{k,\ell}_{\bm{\beta}_i}(\Omega_i)$ satisfying $\varphi =\chi |_{\gamma}$, and we set
\begin{align}
    \|\varphi\|_{H^{k-1/2,l-1/2}_{\bbeta_i}(\gamma)}&= \inf_{\chi\in H^{k,l}_{\bm{\beta}_i}(\Omega_i),\; \chi|_{\gamma} =\varphi}\|\chi\|_{H^{k,l}_{\bm{\beta}_i}(\Omega_i)},
\end{align}
where the infimum is taken over all such liftings. We also introduce the subspace of traces that vanish on the complement of $\gamma$
\begin{align}
    H^{1/2,1/2}_{\bbeta_i,00}(\gamma;\Omega_i)&:=\{v\in H^{1/2,1/2}_{\bbeta_i}(\gamma): \exists w\in H^{1,1}_{\bbeta_i}(\Omega_i)\;\mathrm{s.t}\;w|_{\gamma}=v \;\mathrm{and}\;w|_{\partial\Omega_i\setminus{\gamma}}=0  \},
\end{align}
equipped with the norm inherited from $H^{1/2,1/2}_{\bbeta_i}(\gamma)$, {see \cite[Section 4]{babuvska1989regularity}}. 

For classical Sobolev spaces, we similarly define
\begin{equation}
        H_{00}^{1/2}(\gamma;\Omega_i)=\{u\in H^{1/2}(\gamma):\exists w\in H^1(\Omega_i)\;\mathrm{s.t.}\;w|_{\gamma}=u\;\mathrm{and}\;w|_{\partial\Omega_i\setminus\gamma}=0\}.
    \end{equation}
    equipped with the norm inherited from $H^{1/2}(\gamma)$.
\begin{remark}
    Although the space $\H^{m-1/2,l-1/2}_{\bm{\beta}_f}(\gamma)$ (resp., $\H^{m-1/2,l-1/2}_{\bm{\beta}_s}(\gamma)$) is intrinsically characterized only by the weights $\beta_{f,i}$ (resp., $\beta_{s,i}$) associated with the vertices of the edges constituting $\gamma$, our definition constructs it based on the full weight parameter $\bm{\beta}_f$ (resp., $\bm{\beta}_s$) of the domain $\Omega_f$ (resp., $\Omega_s$). 
\end{remark}

\subsection{Weak solutions}
Throughout this paper, $C$ denotes a positive constant, depending only on the domains $\Omega$, $\Omega_f$ and $\Omega_s $, whose value may change from line to line. We are now in a position to introduce a weak formulation of the system \eqref{eq:stoke_elasticity_t}--\eqref{con:initial} .
	
	Assume  that 
	\begin{equation}\label{eq:initial_assume}
		\begin{cases}
			&\f \in L^2(0,T:\L^2(\Omega_f)),\quad \h \in L^2(0,T:\L^2(\Omega_s)),\; \\
            &\u_{in}\in L^\infty (0,T:\H^{1/2}(\Gamma_{f})),\quad \partial_t \u_{in}\in L^\infty (0,T:\H^{1/2}(\Gamma_{f})),\;\\
			&\u_0 \in \H^1(\Omega_f), \quad \e_0\in \H^1(\Omega_s),\quad \e_1\in \H^1(\Omega_s).
		\end{cases}
	\end{equation}
	Moreover, assume that the following compatibility conditions hold at the initial time:
	\begin{equation}\label{eq:initial_compatibility}
		\begin{cases}
			\nabla\cdot \u_0=0& \text{on} \,\,\Omega_f,\\
			\u_0=\u_{in}|_{t=0} &\text{on} \,\, \Gamma_{f},\\
			\e_1=\u_0  &\text{on} \,\,\Gamma,
		\end{cases}
	\end{equation}

    	We denote by $(\cdot,\cdot)_{\mathcal{D}}$ the $L^2$ inner product on a domain $\mathcal{D}$.
        Define the bilinear forms
	\begin{align*}
		a_1[\u,\vv ]&=2(\mathbb{D}(\u), \mathbb{D}(\vv ))_{\Omega_f}&&\forall\u ,\vv \in \H^1(\Omega_f),\\
		b[\vv,q]&=-(\nabla\cdot\vv,q)_{\Omega_f}  &&\forall\vv \in \H^1(\Omega_f),\, q\in L^2(\Omega_f),\\
		a_2[\u,\vv ]&=2\mu _1(\mathbb{D}(\u), \mathbb{D}(\vv ))_{\Omega_s}+\ur (\nabla\cdot\u ,\nabla\cdot\vv)_{\Omega_s}  &&\forall\u ,\vv \in \H^1(\Omega_s).
	\end{align*}

    To derive the variational formulation of system \eqref{eq:stoke_elasticity_t}--\eqref{con:initial}, we introduce the  space
    \begin{align*}
        \H^1_{\Gamma_f}(\Omega)&=\{\vv\in \H^1(\Omega):\vv|_{\Gamma_f}=\mathbf{0}\},
    \end{align*}
    which incorporates the boundary condition \eqref{con:boundary}.
    
We are now ready to define weak solutions.
    \begin{definition}[Weak solution]
        We say that functions
	$$\u\in L^2(0,T;\H^1(\Omega_f)\quad \text{with} \quad \partial_t\u\in L^2(0,T;\L^2(\Omega_f)),\quad  p \in L^2(0,T;L^2(\Omega_f))$$
	and 
	$$\e\in L^2(0,T;\H^1(\Omega_s))\quad \text{with}\quad  \partial_t\e \in L^2(0,T;\H^1(\Omega_s)),\quad \partial_{tt}\e\in L^2(0,T;\L^2(\Omega_s))$$
	constitute a weak solution of \eqref{eq:stoke_elasticity_t}--\eqref{con:initial} provided that the following conditions are satisfied:
	\begin{itemize}
	    \item [(i)] For every $ \vv \in \H^1_{\Gamma_f}(\Omega)$, $ \psi  \in L^2(\Omega_f)$ and for a.e. time $t\in (0,T]$,
        \begin{equation}\label{eq:weak_formulation_t}
		\begin{cases}
			(\partial_t\u ,\vv )_{\Omega_f}+a_1[\u,\vv ]+b[\vv,p]+(\partial_{tt}\e,\vv )_{\Omega_s}+a_2[\e,\vv ]=(\f,\vv )_{\Omega_f}+(\h ,\vv )_{\Omega_s},&\\
			b[\u ,\psi ]=0,&\\
            \u =\u _{in} \;\;\text{on} \;\;\Gamma_f.
		\end{cases}
	\end{equation}
    \item[(ii)] The initial conditions hold:
\[
\u(0) = \u_0, 
\qquad
\e(0) = \e_0,
\qquad
\partial_t \e(0) = \e_1.
\]
	\end{itemize}
    \end{definition}
    

    \subsection{Main results}
    We begin by observing that, for each fixed $t \in (0,T]$, the pressure $p(t,\x)$ formally satisfies the following elliptic problem in the spatial variable $\x$:
    \begin{equation}\label{pde:p_u_eta}
		\begin{cases}
			\Delta p={\nabla\cdot}(-\nabla\cdot (2\mathbb{D}(\mathbf{u}) - p \mathbb{I}))&\text{in}\,\,\Omega_f\times (0,T],\\
			p=[2\D (\u)\n-\sigma_s(\e)\n]\cdot \n &\text{on}\,\,\Gamma\times (0,T],\\
            \frac{\partial p}{\partial \n }=[-\nabla\cdot (2\mathbb{D}(\mathbf{u}) - p \mathbb{I})]\cdot \n +[\nabla \cdot 2\D (\u)]\cdot \n &\text{on} \,\,
			\Gamma_f\times (0,T].
		\end{cases}
	\end{equation}
The boundary condition on $\Gamma$ follows by taking the normal component of the traction continuity relation in \eqref{con:interface}.

 We next decompose the pressure into $p = r + q$, where $r$ and $q$ solve the following boundary value problems:
\begin{align}
    &\begin{cases}
			\Delta r={\nabla\cdot} (-\nabla\cdot (2\mathbb{D}(\mathbf{u}) - p \mathbb{I})) &\text{in}\,\,\Omega_f\times (0,T],\\
			r=0&\text{on}\,\,\Gamma\times (0,T],\\
			\frac{\partial r}{\partial \n }=[-\nabla\cdot (2\mathbb{D}(\mathbf{u}) - p \mathbb{I})]\cdot \n &\text{on} \,\,
			\Gamma_f\times (0,T],
		\end{cases} \label{eq:pressure_r}\\
        &\begin{cases}
			\Delta q=0 &\text{in}\,\,\Omega_f\times (0,T],\\
			q=[2\D (\u)\n-\sigma_s(\e)\n]\cdot \n &\text{on}\,\,\Gamma\times (0,T],\\
			\frac{\partial q}{\partial \n }=[\nabla \cdot 2\D (\u)]\cdot \n &\text{on} \,\,
			\Gamma_f\times (0,T],
		\end{cases}.
        \label{eq:pressure_q}
\end{align}
Assuming \begin{equation}
    -\nabla\cdot (2\mathbb{D}(\mathbf{u}) - p \mathbb{I})\in \L^2(\Omega_f),
\end{equation}
the function $r\in H^1_\Gamma(\Omega_f)$   is the unique weak solution of
\begin{equation}\label{map:definition_M}
    (\nabla r,\nabla \phi)_{\Omega_f}=(-\nabla\cdot (2\mathbb{D}(\mathbf{u}) - p \mathbb{I}),\nabla\phi)_{\Omega_f},\qquad\forall \phi \in H^1_\Gamma(\Omega_f),
\end{equation}
where
\begin{align}
        H^1_{\Gamma}(\Omega_f)&=\{v\in H^1(\Omega_f):v|_{\Gamma}=0\}.
    \end{align}
Consequently, $q=p-r$ can be expressed entirely in terms of $\u$ and $\e$. Therefore, the full FSI solution$(\u,p,\e)$ of  \eqref{eq:stoke_elasticity_t} may be characterized solely through the velocity-displacement pair $(\u,\e)$. 
    
We shall show that  $(\u,\e)$ belongs to the weighted Sobolev space
\begin{equation*}
    \H^{2,2}_{\bbeta_f}(\Omega_f)\times\H^{2,2}_{\bbeta_s}(\Omega_s),
\end{equation*}
and that this regularity is inherited from the initial data $(\u_0,\e_0)$.

To identify the appropriate weighted framework, note that the Stokes and elasticity operators formally satisfy
\begin{align*}
    -\nabla\cdot (2\mathbb{D}(\mathbf{u}) - p \mathbb{I})&=\f-\partial_t\u \qquad \text{with}\quad  \f\in L^2(0,T;\L^2(\Omega_f)),\\
    -\nabla\cdot (2\mu_1\mathbb{D}(\e ) +\mu_2(\nabla\cdot \e ) \mathbb{I})&=\h-\partial_{tt}\e\qquad\text{with}\quad  \h \in L^2(0,T;\L^2(\Omega_s)).
\end{align*}
    
  Motivated by these considerations, we introduce the solution space $\X_{\bbeta_f,\bbeta_s}$, consisting of all pairs $(\u,\e)$ satisfying:
    \begin{itemize}
        \item Interior regularity:
        \begin{equation}
            \begin{cases}
            \u &\in \H^{2,2}_{\bbeta_f}(\Omega_f), \quad \text{with}\; \nabla \cdot \u =0, \\
            \e &\in \H^{2,2}_{\bbeta_s}(\Omega_s),\quad \text{with} \;-\nabla\cdot (2\mu_1\mathbb{D}(\e ) +\mu_2(\nabla\cdot \e ) \mathbb{I})\in \L^2(\Omega_s).\\
        \end{cases}
        \end{equation}
        \item Existence of an  associated  pressure $p$: There exists $p\in H^{1,1}_{\bbeta_f}(\Omega_f)$ such that
        \begin{equation}\label{eq:con_p}
            \begin{cases}
                &-\nabla\cdot (2\mathbb{D}(\mathbf{u}) - p \mathbb{I})\in \L^2(\Omega_f),\\
                &(2 \mathbb{D}(\mathbf{u}) - p\mathbb{I})\mathbf{n} = \si_s(\e ) \mathbf{n}\qquad\text{on}\,\,\Gamma.
            \end{cases}
        \end{equation}
        \item Construction of a modified pressure $q$: Let $r\in H^1_\Gamma(\Omega_f)$ be the weak soultion of \eqref{map:definition_M}, and define 
        \begin{equation}\label{def:q}
            q=p-r.
        \end{equation}
        Then  $({\u},q)$ satisfies the Stokes system
        \begin{equation}\label{eq:q_stokes}
		\begin{cases}
			-\nabla\cdot (2\mathbb{D}({\u}) - q \mathbb{I})=-\nabla\cdot (2\mathbb{D}(\mathbf{u}) - p \mathbb{I})-\nabla r &\text{in}\,\,\Omega_f,\\
			\nabla\cdot {\u}=0 &\text{in}\,\,\Omega_f,\\
			(2\mathbb{D}({\u}) - q \mathbb{I})\n =\si_s(\e)\n&\text{on} \,\,\Gamma,\\
            {\u}=\u|_\Gamma &\text{on} \,\,\Gamma_f,
		\end{cases}
	\end{equation}
    together with the weak formulation
    \begin{equation}\label{eq:weak_solution_q}
        \begin{cases}
            &a_1[\u,\vv ]+b[\vv,q]=(\si_s(\e)\n,\vv)_\Gamma+(-\nabla\cdot (2\mathbb{D}(\mathbf{u}) - p \mathbb{I}),\vv )_{\Omega_f}-(\nabla r,\vv )_{\Omega_f},\qquad\forall \vv \in \H ^1_{\Gamma_f}(\Omega_f),\\
            &b[\u ,\psi ]=0,\qquad\forall \psi \in L^2(\Omega_f),\\
        \end{cases}
    \end{equation}
        \end{itemize} 
    where
\begin{align}
        \H^1_{\Gamma_f}(\Omega_f)&=\{v\in \H^1(\Omega_f):v|_{\Gamma_f}=0\}.
    \end{align}

For the analysis that follows, we introduce the relevant function spaces. For $i=f,s$,   we define the vector-valued space
\[\H(\mathrm{div},\Omega_i):=\{\vv\in \L^2(\Omega_i):\nabla\cdot \vv \in L^2(\Omega_i)\}\]
    equipped with the norm 
    $\|\vv\|_{\H(\mathrm{div},\Omega_i)}=\|\vv \|_{\L^2(\Omega_i)}+\|\nabla\cdot \vv \|_{L^2(\Omega_i)}$.
    Of particular importance is its divergence-free subspace
    \[\H(\mathrm{div}^0,\Omega_i):=\{\vv\in \L^2(\Omega_i):\nabla\cdot \vv =0\}.\]
The corresponding matrix space {$\mathbb{H}(\mathrm{div},\Omega_i)$} 
consists of all  $\mathbb{R} ^{2\times2} $ matrices whose row vectors belong to $\H(\mathrm{div},\Omega_i)$. Its norm is taken as the sum of the $\H(\mathrm{div},\Omega_i)$-norms of its row vectors.

We now show that the modified pressure $q=p-r$ defined in \eqref{def:q} is uniquely determined by $(\u,\e)$ and is independent of the particular choice of $(p,r)$. 

Suppose that two triples $(p_1,r_1,q_1)$ and $(p_2,r_2,q_2)$ satisfy \eqref{eq:con_p}--\eqref{eq:weak_solution_q}.
By \cite[Lemma,2.2]{nguyen2015boundary}, the space $\L^2(\Omega_f)$ admits the  orthogonal decomposition 
\begin{equation}
    \L^2(\Omega_f)=\{\mathbf{z}\in \H(\mathrm{div}^0,\Omega_f):\mathbf{z}\cdot \n =0\;\mathrm{on}\;\Gamma_f\}\oplus\nabla H^1_\Gamma(\Omega_f).
\end{equation}
Consequently, for any $ \vv \in \L^2(\Omega_f)$, there exist
\begin{equation}
    \vv_1\in \{\mathbf{z}\in \H(\mathrm{div}^0,\Omega_f):\mathbf{z}\cdot \n =0\;\mathrm{on}\;\Gamma_f\},\qquad v_2\in H^1_\Gamma(\Omega_f),
\end{equation}
such that $\vv=\vv_1+\nabla v_2$. Substituting this decomposition into \eqref{eq:weak_solution_q} and invoking the definition of $r$ in \eqref{map:definition_M}, we obtain
\begin{equation}
    b[\vv,q_1]-b[\vv,q_2]=0.
\end{equation}
By the inf-sup condition established in \cite[Lemma 3.1]{bermudez1998finite} and  \cite[Lemma 4.1]{girault2012finite}, this implies that $q_1=q_2$. Furthermore, when $(\u,\e)=\mathbf{0}$, one may choose $p=0$, which consequently forces $q=0$.

Finally, we equip  $\X_{\bbeta_f,\bbeta_s}$  with the norm:
\begin{align}
    \|(\u,\e)\|_{\X_{\bbeta_f,\bbeta_s}}&:=\|\u\|_{\H^{2,2}_{\bbeta_f}(\Omega_f)}+\|-\nabla\cdot (2\mathbb{D}(\mathbf{u}) - q \mathbb{I})\|_{\L^2(\Omega_f)}\notag\\
    &\quad +\|\e \|_{\H^{2,2}_{\bbeta_s}(\Omega_s)}+\|-\nabla\cdot (2\mu_1\mathbb{D}(\e ) +\mu_2(\nabla\cdot \e ) \mathbb{I})\|_{\L^2(\Omega_s)}.
\end{align}
where $q$ actually  depends only on $(\u,\e)\in \X_{\bbeta_f,\bbeta_s}$.

    To determine suitable weight exponents $\bbeta_f,\bbeta_s$ that compensate for the loss of regularity near corner singularities and ensure $\H^{2,2}_{\bbeta}$-regularity, we introduce the following characteristic constants associated with the corner angles of the domain:
    \begin{align*}
        \kappa_{f,j}&=\min \;\{|\operatorname{Im}(\lambda)|:\cos ^2(z\omega_{f,j})=z^2\sin^2(\omega_{f,j}), \,\lambda=iz
    ,\,\ z \neq 0 \}\quad \text{for}\quad j=0,J_f,\\
    \kappa_{f,j}&=\min \;\{|\operatorname{Im}(\lambda)|:\sin ^2(z\omega_{f,j})=z^2\sin^2(\omega_{f,j}), \,\lambda=iz,\,\ z \neq 0 
    \}\quad \text{for}\quad j=1,\cdots,J_f-1,\\
    \kappa_{s,j}&=\min \;\{|\operatorname{Im}(\lambda)|:\sin ^2 (z\omega_{s,j})=\frac{4(1-\nu)^2}{3-4\nu}-\frac{z^2}{3-4\nu}\sin ^2 (\omega_{s,j}),\,\lambda=iz,\,\ z \neq 0 
    \}\quad \text{for}\quad j=0,J_s,\\
    \kappa_{s,j}&=\min \;\{|\operatorname{Im}(\lambda)|:\sin ^2(z\omega_{s,j})=z^2\sin^2(\omega_{s,j}),\,\lambda=iz,\,\ z \neq 0 
    \}\quad \text{for}\quad j=1,\cdots,J_s-1.
    \end{align*}
    Here the Lamé constants satisfy the relation  $\ur/\ue=2\nu/(1-2\nu)$.
    
The quantities $\kappa_{f,j}$ and $\kappa_{s,j}$ arise from the eigenvalue analysis of the Stokes and elasticity systems posed in an infinite sector of opening angle $\omega$. They characterize the leading-order singular behavior of solutions near geometric corners. Details of this analysis are provided in Section~\ref{sec:solution_operators}.
 
     The principal theoretical results of this work are summarized in the following existence, uniqueness, and weighted regularity theorem.
    \begin{theorem}[Weighted $\H^2$ regularity]\label{thm:weighted_H2} 
    Let the weight vectors $\bm{\beta}_f = (\beta_{f,0}, \ldots, \beta_{f,J_f})$ and $\bm{\beta}_s = (\beta_{s,0}, \ldots, \beta_{s,J_s} )$ be chosen such that:
\begin{align*}
& 0 < \beta_{f,j} < 1, \quad \beta_{f,j} > 1 - \kappa_{f,j},\quad 0 < \beta_{s,j} < 1, \quad \beta_{s,j} > 1 - \kappa_{s,j},
\end{align*}
with the compatibility conditions $ \beta_{f,0}=\beta_{s,0}$ and $\beta_{f,J_f}=\beta_{s,J_s}$.
        	Assume that $\f, \h ,\u_{in}, \u_0 , \e_0$ and $\e_1 $ satisfy \eqref{eq:initial_assume}--\eqref{eq:initial_compatibility} and, in addition,
            \begin{align*}
                &(\u_0,\e_0)\in \X_{\bbeta_f,\bbeta_s},\quad \f\in L^\infty(0,T;\L^2(\Omega_f)),\quad\partial_t\f\in L^2(0,T;\L^2(\Omega_f)),\quad \partial_t\h\in L^2(0,T;\L^2(\Omega_s)),\\
                &\u_{in}\in L^\infty(0,T;\H^{3/2,3/2}_{\bbeta_f}(\Gamma_f)),\quad \partial_{t}\u_{in}\in L^\infty(0,T;\H^{1/2}(\Gamma_f)),\quad \partial_{tt}\u_{in}\in L^2(0,T;\H^{1/2,1/2}_{\bbeta_f,00}(\Gamma_f;\Omega_f)^\prime ).
            \end{align*}
             Then, there exists a unique triplet $({\u},{p},{\e})$ solving the coupled system \eqref{eq:stoke_elasticity_t}--\eqref{con:initial} such that
		\begin{align*}
			&{\u }\in L^\infty(0,T;\L^2(\Omega_f))\cap L^2(0,T;\H^1(\Omega_f)),\quad {\e} \in L^\infty(0,T;\H^1(\Omega_s)),\quad {p}\in L^2(0,T;L^2(\Omega_f)),\\
			&\partial_t{\u }\in L^\infty(0,T;\L^2(\Omega_f))\cap L^2(0,T;\H^1(\Omega_f)),	\quad \partial_t{\e} \in L^\infty(0,T;\H^1(\Omega_s)),\quad \partial_{tt}{\e} \in L^\infty(0,T;\L^2(\Omega_s)).		
		\end{align*}
        Moreover, for almost every $t\in[0,T]$, the solution enjoys higher regularity in the weighted spaces:
        \begin{equation}
            \u(t)\in \H^{2,2}_{\bbeta_f}(\Omega_f) ,\quad p(t)\in H^{1,1}_{\bbeta_f}(\Omega_f),\quad \e (t)\in \H^{2,2}_{\bbeta_s}(\Omega_s).
        \end{equation}
        The following estimate holds:
        \begin{align}\label{eq:estimate_u}
			&\|(\u(t),\e(t))\|_{\X_{\bbeta_f,\bbeta_s}}+\|p(t)\|_{{H}^{1,1}_{\bm{\beta}_f}(\Omega_f)}+\|{\u }(t)\|_{\L^2(\Omega_f)}+\|{\u }\|_{L^2(0,T;\H^1(\Omega_f)}+\|{\e }(t)\|_{\H^1(\Omega_s)}\notag\\
			&+ \|{p}\|_{L^2(0,T;L^2(\Omega_f))}+\|\partial_t{\u }(t)\|_{\L^2(\Omega_f)}+\|\partial_t{\u }\|_{L^2(0,T;\H^1(\Omega_f)}+\|\partial_t{\e }(t)\|_{\H^1(\Omega_s)}+\|\partial_{tt}{\e }(t)\|_{\L^2(\Omega_s)}\notag\\
			&\leqslant Ce^{CT}\big(\|\f\|_{L^\infty(0,T;\L^2(\Omega_f))}+\|\partial_t\f\|_{L^2(0,T;\L^2(\Omega_f))}+\|\h\|_{H^1(0,T;\L^2(\Omega_s))}+\|\u_{in}\|_{L^\infty(0,T;\H^{3/2,3/2}_{\bbeta_f}(\Gamma_f))}\notag\\
        &\quad+\|\partial_t \u_{in}\|_{L^\infty(0,T;\H^{1/2}(\Gamma_f))} +\|\partial_{tt} \u_{in}\|_{L^2(0,T;\H^{1/2,1/2}_{\bbeta_f,00}(\Gamma_{f};\Omega_f)^\prime )}+\|\e_1\|_{\H^1(\Omega_s)}+\|(\u_0,\e_0)\|_{X_{\bbeta_f,\bbeta_s}}\big).
		\end{align}
    \end{theorem}

\section{Proof of the Main Results}\label{sec:proof}
In this section, we establish $\H^{2,2}_{\bbeta}$-regularity for the fluid-structure interaction system posed in polygonal domains with interface-boundary corner singularities.

The proof is based on an operator framework that decouples the Stokes and elasticity subproblems while preserving their coupling through interface operators. This reformulation reduces the coupled system to an evolution equation in the weighted space $\mathbf{H}^{2,2}_{\boldsymbol{\beta}}$. The desired regularity then follows from semigroup theory together with the weighted elliptic estimates developed in \cite{guo1993regularity,guo2006analytic}.

    \subsection{Solution operators}\label{sec:solution_operators}

        We begin by introducing a collection of solution operators that map boundary data and source terms to the corresponding solutions of the Stokes problem in $\Omega_f$ and the elasticity problem in $\Omega_s$. These operators constitute the core analytical tools used throughout the paper and provide the foundation for the operator-theoretic reformulation of the coupled system.

        \begin{itemize}
            \item [1.] The Neumann Map $N_\Gamma$, $\tilde{N}_\Gamma$: The operator $N_\Gamma$ maps a given displacement field $\mathbf{w}$ defined on $\Omega_s$ to a solution $(\mathbf{v}, q)$ of the Stokes system in $\Omega_f$, subject to a Neumann boundary condition on $\Gamma$ induced by $\mathbf{w}$:
            \begin{equation}
                \label{solution_operator_N}
		(\vv  ,q)=N_\Gamma \w \Longleftrightarrow\begin{cases}
			-\nabla\cdot (2\mathbb{D}(\mathbf{v}) - q \mathbb{I})=\mathbf{0} &\text{in}\,\,\Omega_f;\\
			\nabla\cdot \mathbf{v}=0 &\text{in}\,\,\Omega_f;\\
			\vv =\mathbf{0}&\text{on}\,\,\Gamma_f;\\
			(2\mathbb{D}(\mathbf{v}) - q \mathbb{I})\mathbf{n}=(2\mu_1\mathbb{D}(\w ) +\mu_2(\nabla\cdot \w ) \mathbb{I})\n &\text{on}\,\,\Gamma;
		\end{cases}
            \end{equation}
            We also introduce the modified Neumann Map $\tilde{N}_\Gamma$, defined by 
            \begin{equation}
                \label{solution_operator_~N}
		(\vv  ,q)=\tilde{N}_\Gamma g \Longleftrightarrow\begin{cases}
			-\nabla\cdot (2\mathbb{D}(\mathbf{v}) - q \mathbb{I})=\mathbf{0} &\text{in}\,\,\Omega_f;\\
			\nabla\cdot \mathbf{v}=0 &\text{in}\,\,\Omega_f;\\
			\vv =\mathbf{0}&\text{on}\,\,\Gamma_f;\\
			(2\mathbb{D}(\mathbf{v}) - q \mathbb{I})\mathbf{n}=g &\text{on}\,\,\Gamma;
		\end{cases}
            \end{equation}
        If $g$ is defined throughout the solid domain  $\Omega_s$, the operator is understood to act via its trace on the interface, i.e., $\tilde{N}_\Gamma g:=\tilde{N}_\Gamma(g{|_\Gamma})$.
        \item[2.] The Dirichlet Map (on $\Gamma_f$) $D_{\Gamma_f}$: 
        This operator maps a given velocity $\mathbf{w}$ on the boundary $\Gamma_f$ to the solution $(\mathbf{v}, q)$ of a homogeneous Stokes problem in $\Omega_f$ with that Dirichlet data:
        \begin{equation}\label{solution_operator_Df}
            (\vv  ,q)=D_{\Gamma_f} \w \Longleftrightarrow\begin{cases}
			-\nabla\cdot (2\mathbb{D}(\mathbf{v}) - q \mathbb{I})=\mathbf{0} &\text{in}\,\,\Omega_f;\\
			\nabla\cdot \mathbf{v}=0 &\text{in}\,\,\Omega_f;\\
			\vv =\w &\text{on}\,\,\Gamma_f;\\
			(2\mathbb{D}(\mathbf{v}) - q \mathbb{I})\mathbf{n}=\mathbf{0} &\text{on}\,\,\Gamma;
		\end{cases}
        \end{equation}
        \item[3.] The Force Map (for Stokes) $F_1$: 
        This operator maps a body force $\psi$ in $\Omega_f$ to the solution $(\mathbf{v}, q)$ of the resulting Stokes problem with homogeneous Dirichlet data on $\Gamma_f$ and homogeneous Neumann data on $\Gamma$:
        \begin{equation}\label{solution_operator_F1}
            (\vv  ,q)=F_1\psi  \Longleftrightarrow\begin{cases}
			-\nabla\cdot (2\mathbb{D}(\mathbf{v}) - q \mathbb{I})=\psi &\text{in}\,\,\Omega_f;\\
			\nabla\cdot \mathbf{v}=0&\text{in}\,\,\Omega_f;\\
			\vv =\mathbf{0}&\text{on}\,\,\Gamma_f;\\
			(2\mathbb{D}(\mathbf{v}) - q \mathbb{I})\mathbf{n}=\mathbf{0} &\text{on}\,\,\Gamma;
		\end{cases}
        \end{equation}
        \item[4.] The Dirichlet Map (on $\Gamma$) $D_\Gamma$: 
        This operator maps a given velocity $\mathbf{v}$ on the interface $\Gamma$ to the solution $\mathbf{w}$ of a homogeneous elasticity problem in $\Omega_s$ with that Dirichlet data on $\Gamma$ and  homogeneous Neumann data on $\Gamma_s$:
        \begin{equation}\label{eq:D_Gamma}
            \w =D_\Gamma \vv \Longleftrightarrow\begin{cases}
			-\nabla\cdot (2\mu_1\mathbb{D}(\w ) +\mu_2(\nabla\cdot \w ) \mathbb{I})=\mathbf{0}&\text{in} \,\,\Omega_s;\\
			\w =\vv &\text{on}\,\,\Gamma;\\
			(2\mu_1\mathbb{D}(\w ) +\mu_2(\nabla\cdot \w ) \mathbb{I})\n=\mathbf{0} &\text{on}\,\,\Gamma_s.
		\end{cases}
        \end{equation}
        If  the input $\vv$ is defined throughout the fluid domain  $\Omega_f$, the operator acts by taking its trace on the interface, i.e., $D_\Gamma \vv:=D_\Gamma(\vv{|_\Gamma})$.
        \item[5.] The Force Map (for Elasticity) $F_2$:\\
        This operator maps a body force $\phi$ in $\Omega_s$ to the solution $\mathbf{w}$ of the resulting elasticity problem with homogeneous Dirichlet data on  $\Gamma$ and  homogeneous Neumann data on $\Gamma_s$:
        \begin{equation}
            \label{solution_operator_F_2}
		\w =F_2\phi \Longleftrightarrow\begin{cases}
			-\nabla\cdot (2\mu_1\mathbb{D}(\w ) +\mu_2(\nabla\cdot \w ) \mathbb{I})=\phi &\text{in} \,\,\Omega_s;\\
			\w =\mathbf{0} &\text{on}\,\,\Gamma;\\
			(2\mu_1\mathbb{D}(\w ) +\mu_2(\nabla\cdot \w ) \mathbb{I})\n =\mathbf{0} &\text{on}\,\,\Gamma_s.
		\end{cases}
        \end{equation}
        \end{itemize}

   Applying the standard Lax–Milgram argument for elliptic regularity along with Theorem 3.1 from \cite{boyer2025study} yields the following continuity results :
    \begin{equation}\label{bounded_map_H^1}
        \begin{aligned}
            &\tilde{N}_\Gamma\in \mathcal{L}(\H^{1/2}_{00}(\Gamma;\Omega_f)^\prime ;\H^1(\Omega_f)\times L^2(\Omega_f));\\
            &D_{\Gamma_f}\in \mathcal{L}(\H^{1/2}(\Gamma_f);\H^1(\Omega_f)\times L^2(\Omega_f));\\
            &D_\Gamma \in  \mathcal{L}(\H^{1/2}(\Gamma);\H^1(\Omega_s)).
        \end{aligned}
    \end{equation}

    The spectral properties of the solution operators associated with the Stokes and elasticity systems are governed by specific transcendental equations. These equations, derived from an eigenvalue analysis of the corresponding boundary value problems in an infinite sector of angle $\omega$, determine the singular behavior of solutions near geometric corners. According to the established theory \cite{guo1993regularity,guo2006analytic}, the eigenvalue parameter $\lambda = iz$ satisfies the following characteristic equations for a given angle $\omega \in (0, 2\pi)$:
    \begin{itemize}
        \item [(1)]Stokes Equations with  Dirichlet or Neumann Conditions \cite[Eq.~(4.33a)]{guo2006analytic}:
        \begin{equation}
            \sin ^2(z\omega)=z^2\sin^2(\omega).
        \end{equation}
        \item [(2)] Stokes Equations with Mixed Boundary Conditions \cite[Eq.~(4.33b)]{guo2006analytic}:
        \begin{equation}
            \cos ^2(z\omega)=z^2\sin^2(\omega).
        \end{equation}
        \item[(3)]  Elasticity Equations with Neumann Conditions \cite[Eq.~(4.13)]{guo1993regularity}:
        \begin{equation}
            \sin ^2(z\omega)=z^2\sin^2(\omega).
        \end{equation}
        \item [(4)]Elasticity Equations with Mixed Boundary Conditions \cite[Eq.~(4.14)]{guo1993regularity}:
        \begin{equation}
            \sin ^2 (z\omega)=\frac{4(1-\nu)^2}{3-4\nu}-\frac{z^2}{3-4\nu}\sin ^2 (\omega)\quad \text{with} \quad \ur/\ue=2\nu/(1-2\nu).
        \end{equation}
    \end{itemize}

    Let $\kappa$ denote the smallest positive imaginary part among the non-zero eigenvalues $\lambda = iz$ with $\operatorname{Im}(\lambda)> 0$. More precisely, for each corner of the domain we define the constants
    \begin{align*}
        \kappa_{f,j}&=\min \;\{|\operatorname{Im}(\lambda)|:\cos ^2(z\omega_{f,j})=z^2\sin^2(\omega_{f,j}), \,\lambda=iz
    ,\,\ z \neq 0 \}\quad \text{for}\quad j=0,J_f,\\
    \kappa_{f,j}&=\min \;\{|\operatorname{Im}(\lambda)|:\sin ^2(z\omega_{f,j})=z^2\sin^2(\omega_{f,j}), \,\lambda=iz,\,\ z \neq 0 
    \}\quad \text{for}\quad j=1,\cdots,J_f-1,\\
    \kappa_{s,j}&=\min \;\{|\operatorname{Im}(\lambda)|:\sin ^2 (z\omega_{s,j})=\frac{4(1-\nu)^2}{3-4\nu}-\frac{z^2}{3-4\nu}\sin ^2 (\omega_{s,j}),\,\lambda=iz,\,\ z \neq 0 
    \}\quad \text{for}\quad j=0,J_s,\\
    \kappa_{s,j}&=\min \;\{|\operatorname{Im}(\lambda)|:\sin ^2(z\omega_{s,j})=z^2\sin^2(\omega_{s,j}),\,\lambda=iz,\,\ z \neq 0 
    \}\quad \text{for}\quad j=1,\cdots,J_s-1.\\
    \end{align*}
These quantities determine the admissible range of weight exponents in the weighted Sobolev spaces and thus control the corner regularity of the coupled FSI system.
    
  \begin{lemma}\label{lem:solution_operator}
        Let the weight vectors $\bm{\beta}_f = (\beta_{f,0}, \ldots, \beta_{f,J_f})$ and $\bm{\beta}_s = (\beta_{s,0}, \ldots, \beta_{s,J_s} )$ be chosen such that:
\begin{align*}
& 0 < \beta_{f,j} < 1, \quad \beta_{f,j} > 1 - \kappa_{f,j}, \
& 0 < \beta_{s,j} < 1, \quad \beta_{s,j} > 1 - \kappa_{s,j},
\end{align*}
with the compatibility conditions $ \beta_{f,0}=\beta_{s,0}$ and $\beta_{f,J_f}=\beta_{s,J_s}$. Then, the solution operators defined  in \eqref{solution_operator_N}--\eqref{solution_operator_F_2} are bounded linear maps between the following  weighted Sobolev spaces:
        \begin{align*}
        N_\Gamma&\in \mathcal{L}(\H^{2,2}_{\bbeta_s}(\Omega_s);\H^{2,2}_{\bbeta_f}(\Omega_f)\times H^{1,1}_{\bbeta_f}(\Omega_f)),\\
        D_{\Gamma_f}&\in \mathcal{L}(\H^{3/2,3/2}_{\bbeta_f}(\Gamma_f);\H^{2,2}_{\bbeta_f}(\Omega_f)\times H^{1,1}_{\bbeta_f}(\Omega_f)),\\
        F_1&\in \mathcal{L}(\bm{L}^2_{\bbeta_f}(\Omega_f);\H^{2,2}_{\bbeta_f}(\Omega_f)\times H^{1,1}_{\bbeta_f}(\Omega_f)),\\
		D_\Gamma&\in \mathcal{L}(\H^{2,2}_{\bbeta_f}(\Omega_f);\H^{2,2}_{\bbeta_s}(\Omega_s)),\\
		F_2&\in \mathcal{L}(\bm{L}^2_{\bbeta_s}(\Omega_s);\H^{2,2}_{\bbeta_s}(\Omega_s)).
	\end{align*}
    \end{lemma}
\begin{proof}
The results follow directly by applying the regularity theory for the Stokes and elasticity systems in polygonal domains, as established in \cite[Theorem 5.5]{guo2006analytic} and \cite[Theorem 5.2]{guo1993regularity}, to the boundary value problems associated with each solution operator.
\end{proof}
\begin{remark} 
    The compatibility conditions $ \beta_{f,0}=\beta_{s,0}$ and $\beta_{f,J_f}=\beta_{s,J_s}$ ensure  the consistency of the weighted Sobolev spaces across the fluid-structure interface $\Gamma$. More precisely:
    \begin{enumerate}
    \item If a function \(\mathbf{v} \in \H^{k,\ell}_{\boldsymbol{\beta}_f}(\Omega_f) \), then its trace on the interface satisfies $$ \mathbf{v}|_{\Gamma} \in \H^{k-1/2,\ell-1/2}_{\boldsymbol{\beta}_f}(\Gamma),$$ where the trace space is determined by the weights \( \beta_{f,i} \)  associated with the interface vertices. Under the compatibility conditions, this trace space coincides with the corresponding weighted trace space induced from the solid side. Consequently, the trace admits an extension \( \mathbf{w} \in \H^{k,\ell}_{\boldsymbol{\beta}_s}(\Omega_s) \).

    \item Conversely, if \( \mathbf{w} \in \H^{k,\ell}_{\boldsymbol{\beta}_s}(\Omega_s) \), then its trace satisfies  $$ \mathbf{w}|_{\Gamma} \in \H^{k-1/2,\ell-1/2}_{\boldsymbol{\beta}_s}(\Gamma), $$ and, owing to the matching of the weights at the interface vertices, this trace can be extended to a function  \( \mathbf{v} \in \H^{k,\ell}_{\boldsymbol{\beta}_f}(\Omega_f) \).
\end{enumerate}
Therefore, the matching of weights across the interface guarantees the mutual trace-extension property between the weighted Sobolev spaces on $\Omega_f$ and $\Omega_s$,  which is essential for the well-posed coupling of the fluid and structure variables.
\end{remark}

Applying the transposition method (see, e.g., \cite{lions2012nonn}) yields the following result.
\begin{lemma}\label{lem:transition_method}
    Let the weight vectors $\bm{\beta}_f = (\beta_{f,0}, \ldots, \beta_{f,J_f})$ and $\bm{\beta}_s = (\beta_{s,0}, \ldots, \beta_{s,J_s} )$ be chosen such that:
\begin{align*}
& 0 < \beta_{f,j} < 1, \quad \beta_{f,j} > 1 - \kappa_{f,j}, \
& 0 < \beta_{s,j} < 1, \quad \beta_{s,j} > 1 - \kappa_{s,j},
\end{align*}
with the compatibility conditions $ \beta_{f,0}=\beta_{s,0}$ and $\beta_{f,J_f}=\beta_{s,J_s}$. Here, the subscript $(\cdot)_1$ denotes the first component of the solution pair returned by the respective operators. Then, the solution operators $D_{\Gamma_f}$ defined in \eqref{solution_operator_Df} and $D_\Gamma$ defined in \eqref{eq:D_Gamma} satisfy the following mapping properties:
\begin{equation}
    (D_{\Gamma_f})_1\in \mathcal{L}(\H^{1/2,1/2}_{\bbeta_f,00}(\Gamma_f;\Omega_f)^\prime ;\L^2(\Omega_f)),\quad D_\Gamma \in \mathcal{L}(\H^{1/2,1/2}_{\bbeta_s,00}(\Gamma;\Omega_s)^\prime ;\L^2(\Omega_s)).
\end{equation}
\end{lemma}
\begin{proof}
   Let $g \in \H^{1/2,1/2}_{\bbeta_f,00}(\Gamma_f;\Omega_f)^\prime$.
For the Stokes problem $(\vv, q) = D_{\Gamma_f}g$, consider the corresponding adjoint problem:
    \begin{equation}\label{adjoint_stokes}
        \begin{cases}
            -\nabla\cdot (2\mathbb{D}(\Psi ) - \pi \mathbb{I})=\phi &\text{in}\,\,\Omega_f;\\
			\nabla\cdot \Psi=0 &\text{in}\,\,\Omega_f;\\
			\Psi  =\mathbf{0}&\text{on}\,\,\Gamma_f;\\
			(2\mathbb{D}(\Psi ) - \pi  \mathbb{I})\mathbf{n}=\mathbf{0} &\text{on}\,\,\Gamma;
        \end{cases}
    \end{equation}
    where $\phi \in \L^2(\Omega_f)$ satisfies $\nabla \cdot \phi = 0$.

    Similarly, let $h \in \H^{1/2,1/2}_{\bbeta_s,00}(\Gamma;\Omega_s)^\prime$.  
For the elasticity problem $\w = D_\Gamma h$, consider the adjoint system:
    \begin{equation}\label{adjoint_elasticity}
        \begin{cases}
			-\nabla\cdot (2\mu_1\mathbb{D}(\Phi  ) +\mu_2(\nabla\cdot \Phi  ) \mathbb{I})=\xi &\text{in} \,\,\Omega_s;\\
			\Phi  =\mathbf{0} &\text{on}\,\,\Gamma;\\
			(2\mu_1\mathbb{D}(\Phi  ) +\mu_2(\nabla\cdot \Phi ) \mathbb{I})\n=\mathbf{0}  &\text{on}\,\,\Gamma_s.
        \end{cases}
    \end{equation}
    with $\xi \in \L^2(\Omega_s)$.
    
    Since $\L^2(\Omega_f)\subset \L^2_{\bbeta_f}(\Omega_f)$ and $\L^2(\Omega_s)\subset \L^2_{\bbeta_s}(\Omega_s)$, Lemma \ref{lem:solution_operator} implies that  the solution $(\Psi,\pi) =F_1(\phi)$ belongs to $ \H^{2,2}_{\bbeta_f}(\Omega_f)\times H^{1,1}_{\bbeta_f}(\Omega_f)$ and $\Phi=F_2(\xi)$ belongs to $ \H^{2,2}_{\bbeta_s}(\Omega_s) $, with the estimate
    \begin{align}\label{adjoint_stokes_solution}
        \|\Psi\|_{\H^{2,2}_{\bbeta_f}(\Omega_f)}+\|\pi \|_{H^{1,1}_{\bbeta_f}(\Omega_f)}&\leq C\|\phi\|_{\L^2(\Omega_f)},\\
       \|\Phi\|_{\H^{2,2}_{\bbeta_s}(\Omega_s)}&\leq C\|\xi\|_{\L^2(\Omega_s)}.
    \end{align}

   It follows that $\vv = (D_{\Gamma_f})_1(g)$  satisfies
   \begin{align}
       (\vv,\phi)_{\Omega_f}&=-\int_{\Omega_f}\vv \cdot (\nabla\cdot (2\mathbb{D}(\Psi ) - \pi \mathbb{I}))\notag\\
       &=-\int_{\Gamma_f } g \cdot (2\mathbb{D}(\Psi ) - \pi  \mathbb{I})\mathbf{n}=:L(\phi), \qquad \forall \phi \in \L^2(\Omega_f), \ \nabla \cdot \phi = 0,
   \end{align}
   where $(\Psi,\pi)$ solves \eqref{adjoint_stokes}.

   By the trace theorem and \eqref{adjoint_stokes_solution}, we have $(2\mathbb{D}(\Psi ) - \pi  \mathbb{I})\mathbf{n}|_{\Gamma_f}\in \H^{1/2,1/2}_{\bbeta_f,00}(\Gamma_f;\Omega_f)$ and therefore,
    \begin{align}\label{operator_L_stokes}
        |L(\phi)|&\leq C \|g\|_{\H_{\bbeta_f,00}^{1/2,1/2}(\Gamma_f;\Omega_f)^\prime}\big(\|\Psi\|_{\H^{2,2}_{\bbeta_f}(\Omega_f)}+\|\pi \|_{H^{1,1}_{\bbeta_f}(\Omega_f)}\big)\notag\\
        &\leq C \|g\|_{\H_{\bbeta_f,00}^{1/2,1/2}(\Gamma_f;\Omega_f)^\prime}\|\phi\|_{\L^2(\Omega_f)}.
    \end{align}
    By the Riesz representation theorem, there exists a unique $\vv \in \L^2(\Omega_f)$, such that $(\vv,\phi)_{\Omega_f}=L(\phi)$ and 
    \begin{equation}\label{estimate_L2_stokes}
        \|\vv \|_{\L^2(\Omega_f)}\leq \|L\| \leq C \|g\|_{\H_{\bbeta_f,00}^{1/2,1/2}(\Gamma_f;\Omega_f)^\prime}.
    \end{equation}
    Similarly, $\w = D_{\Gamma} h$  satisfies
   \begin{align}
       (\w,\xi )_{\Omega_s}&=-\int_{\Omega_s}\w  \cdot (\nabla\cdot (2\mu_1\mathbb{D}(\Phi  ) +\mu_2(\nabla\cdot \Phi  ) \mathbb{I}))\notag\\
       &=-\int_{\Gamma } h \cdot (2\mu_1\mathbb{D}(\Phi  ) +\mu_2(\nabla\cdot \Phi  ) \mathbb{I})\mathbf{n}=:L_1(\xi),\qquad \forall\xi \in \L^2{(\Omega_s)},
   \end{align}
   where  $\Phi$ solves \eqref{adjoint_elasticity}.   Repeating the same argument as in \eqref{operator_L_stokes}--\eqref{estimate_L2_stokes},  we conclude that there exists a unique $\w \in \L^2(\Omega_s)$, such that $(\w,\xi )_{\Omega_s}=L_1(\xi)$ and 
   \begin{equation}\label{estimate_L2_elasticity}
        \|\w \|_{\L^2(\Omega_s)}\leq  C \|h\|_{\H_{\bbeta_s,00}^{1/2,1/2}(\Gamma;\Omega_s)^\prime}.
    \end{equation}
\end{proof}

Assume the conditions of Lemma \ref{lem:solution_operator}. For any $\phi \in \H^{1/2,1/2}_{\bbeta_f,00}(\Gamma;\Omega_f)$, let $(\Lambda_\phi, \pi_\phi)$ denote the solution of the  Stokes system:
\begin{equation}\label{trace_operator_Df}
            \begin{cases}
			-\nabla\cdot (2\mathbb{D}(\Lambda) - \pi \mathbb{I})=\mathbf{0} &\text{in}\,\,\Omega_f;\\
			\nabla\cdot \Lambda =0 &\text{in}\,\,\Omega_f;\\
			 \Lambda  =\mathbf{0} &\text{on}\,\,\Gamma_f;\\
			 (2\mathbb{D}(\Lambda) - \pi \mathbb{I})\n =\phi &\text{on}\,\,\Gamma.
		\end{cases}
        \end{equation}
    By the result in \cite[Theorem~5.5]{guo2006analytic}, it follows that  $(\Lambda_\phi,\pi_\phi)\in \H^{2,2}_{\bbeta_f}(\Omega_f)\times H^{1,1}_{\bbeta_f}(\Omega_f)$ with
    \begin{equation}\label{test_Lambda}
        \|\Lambda_\phi\|_{ \H^{2,2}_{\bbeta_f}(\Omega_f)}+\|\pi_\phi\|_{H^{1,1}_{\bbeta_f}(\Omega_f)}\leq C\|\phi\|_{\H^{1/2,1/2}_{\bbeta_f,00}(\Gamma;\Omega_f)}.
    \end{equation}
    In fact, defining $\Sigma_\phi := (2\mathbb{D}(\Lambda_\phi) - \pi_\phi \mathbb{I})\n \in \H^{1,1}_{\bbeta_f}(\Omega_f)$, we have $\Sigma_\phi|_{\Gamma_f} \in \H^{1/2,1/2}_{\bbeta_f,00}(\Gamma_f;\Omega_f)$  whenever $\phi \in \H^{1/2,1/2}_{\bbeta_f,00}(\Gamma;\Omega_f)$. To see this, let $\chi$ be a cut off function such that $\chi =1$ on $\Gamma_f$  and $\chi =0$ on $\Gamma$. Then
    \begin{equation}
        \chi \Sigma_\phi =\Sigma_\phi(\chi-1)+\Sigma_\phi \in \H^{1,1}_{\bbeta_f}(\Omega_f) 
    \end{equation}
    with $\chi \Sigma_\phi=0$ on $\Gamma$ and $\chi \Sigma_\phi=\Sigma_\phi|_{\Gamma_f}$ on $\Gamma_f$.

    By Lemma \ref{lem:transition_method}, define $X(D_{\Gamma_f}):=\{(D_{\Gamma_f})_1(g)\in \L^2(\Omega_f): g \in \H^{1/2,1/2}_{\bbeta_f,00}(\Gamma_f;\Omega_f)^\prime\}$,  endowed with the norm
    \[
    \|\vv\|_{X(D_{\Gamma_f})}:=\|\vv\|_{\L^2(\Omega_f)}+\|\vv|_{\Gamma_f}\|_{\H^{1/2,1/2}_{\bbeta_f,00}(\Gamma_f;\Omega_f)^\prime}.
    \]
    
    For $\vv \in X(D_{\Gamma_f})$ we define its trace on $\Gamma $ by transposition:
    \begin{equation}\label{eq:definition_trace_L^2}
        \langle \vv, \phi \rangle_{\H^{1/2,1/2}_{\bbeta_f,00}(\Gamma;\Omega_f)^\prime,\H^{1/2,1/2}_{\bbeta_f,00}(\Gamma;\Omega_f)}:= -\langle \vv|_{\Gamma_f}, (2\mathbb{D}(\Lambda_\phi ) - \pi_\phi  \mathbb{I})\n \rangle_{\H^{1/2,1/2}_{\bbeta_f,00}(\Gamma_f;\Omega_f)^\prime,\H^{1/2,1/2}_{\bbeta_f,00}(\Gamma_f;\Omega_f)}
    \end{equation}
for each $\phi \in \H^{1/2,1/2}_{\bbeta_f,00}(\Gamma;\Omega_f)$.
        
        Consequently, the trace $\vv|_{\Gamma} \in \H^{1/2,1/2}_{\bbeta_f,00}(\Gamma;\Omega_f)'$ is well defined.
We denote this trace operator by $\gamma_{\Gamma}(\vv)$ for $\vv \in X(D_{\Gamma_f})$.

Based on this definition, we have the following result.

\begin{lemma}\label{lem:trace_L2}
    Let the weight vectors $\bm{\beta}_f = (\beta_{f,0}, \ldots, \beta_{f,J_f})$ and $\bm{\beta}_s = (\beta_{s,0}, \ldots, \beta_{s,J_s} )$ be chosen such that:
\begin{align*}
& 0 < \beta_{f,j} < 1, \quad \beta_{f,j} > 1 - \kappa_{f,j}, \
& 0 < \beta_{s,j} < 1, \quad \beta_{s,j} > 1 - \kappa_{s,j},
\end{align*}
with the compatibility conditions $ \beta_{f,0}=\beta_{s,0}$ and $\beta_{f,J_f}=\beta_{s,J_s}$. Then, the linear operator $\gamma_{\Gamma}: X(D_{\Gamma_f})\rightarrow \H^{1/2,1/2}_{\bbeta_f,00}(\Gamma;\Omega_f)^\prime$ is bounded.
\end{lemma}
\begin{proof}
    Let $\vv \in X(D_{\Gamma_f})$. For any $\phi \in \H^{1/2,1/2}_{\bbeta_f,00}(\Gamma;\Omega_f)$, we have
    \begin{align}
        \langle \vv, \phi \rangle_{\H^{1/2,1/2}_{\bbeta_f,00}(\Gamma;\Omega_f)^\prime,\H^{1/2,1/2}_{\bbeta_f,00}(\Gamma;\Omega_f)}&=-\langle \vv|_{\Gamma_f},(2\mathbb{D}(\Lambda_\phi ) - \pi_\phi  \mathbb{I})\n  \rangle_{\H^{1/2,1/2}_{\bbeta_f,00}(\Gamma_f;\Omega_f)^\prime,\H^{1/2,1/2}_{\bbeta_f,00}(\Gamma_f;\Omega_f)}\notag\\
        &\leq \|\vv\|_{\H^{1/2,1/2}_{\bbeta_f,00}(\Gamma_f;\Omega_f)^\prime}\|(2\mathbb{D}(\Lambda_\phi ) - \pi_\phi  \mathbb{I})\n \|_{\H^{1/2,1/2}_{\bbeta_f,00}(\Gamma_f;\Omega_f)}\notag\\
        &\leq \|\vv\|_{X(D_{\Gamma_f})}\|\phi\|_{\H^{1/2,1/2}_{\bbeta_f,00}(\Gamma;\Omega_f)}\hspace{40pt}\text{(\eqref{test_Lambda} is used)}.
    \end{align}
    Hence, $\gamma_{\Gamma}$ is a bounded linear operator.
\end{proof}

    \subsection{Proof of Theorem \ref{thm:weighted_H2}}\label{sec:proof_1}

We first establish the existence of a weak solution to the fluid-structure interaction (FSI) problem \eqref{eq:stoke_elasticity_t}. Our approach follows the strategy of  \cite[Theorem 2.3]{du2004semidiscrete}, where the case of homogeneous Dirichlet boundary data ($\mathbf{u}_{in} = 0$) is treated.  As a first step, we reduce the original problem with nonhomogeneous Dirichlet boundary data $\mathbf{u}_{in}$ to an equivalent formulation with homogeneous Dirichlet conditions.  Building on this reduction, we then apply standard semigroup theory to derive $\H^{2,2}_{\bbeta}$-regularity for the FSI system posed in polygonal domains with interface-boundary corner singularities.

    \medskip
    To proceed, we first recall a preliminary result concerning normal traces.
    
     For a general $\L^2$-function, its trace on the boundary is not well defined.  However, a normal trace can be rigorously defined for functions in $ \H(\mathrm{div},\Omega_i) $. 
     Let $\Gamma_e $ be a subset of $\partial\Omega_i$
 and $\n$ denote the unit outward normal vector. 
 
    For any $\vv \in \H(\mathrm{div},\Omega_i) $, its normal trace $\vv\cdot \n$ on $\Gamma_e$ is defined in the dual sense via Green's formula: for all $\psi\in H_{00}^{1/2}(\Gamma_e;\Omega_i)$,
    \begin{equation}\label{definition:normal_trace}
        \langle \vv  \cdot \n,\psi \rangle _{H_{00}^{1/2}(\Gamma_e;\Omega_i)^\prime ,H^{1/2}_{00}(\Gamma_e;\Omega_i)}:=(\nabla\cdot \vv,w)_{\Omega_i}+( \vv,\nabla w)_{\Omega_i},
    \end{equation}
    where $w\in H^1(\Omega_i)$ is any lifting satisfying
    \begin{equation}\label{lift_w}
        w|_{\Gamma_e}=\psi,\quad w|_{\partial\Omega_i\setminus\Gamma_e}=0,\quad \|w\|_{H^1(\Omega_i)}\leq C\|\psi\|_{H^{1/2}_{00}(\Gamma_e;\Omega_i)}.
    \end{equation}
    
This construction yields the following lemma.
\begin{lemma}\label{lem:gamma_n}
    Let $\vv\in \H(\mathrm{div},\Omega_i)$. Then the normal trace operator $\gamma_{\n;\Gamma_e}(\vv):=\vv\cdot \n\in H^{1/2}_{00}(\Gamma_e;\Omega_i)^\prime$ defined by \eqref{definition:normal_trace} is well defined and bounded. In particular,
    \[\gamma_{{\n;\Gamma_e}}:\H(\mathrm{div},\Omega_i)\rightarrow H^{1/2}_{00}(\Gamma_e;\Omega_i)^\prime\] is a continuous linear operator.
    \end{lemma}
    \begin{proof}
        The well-definedness and boundedness of $\gamma_{\n;\Gamma_e}$ follow by arguments identical to those in \cite[Theorem 1.2]{temam2024navier}.
    \end{proof}

    \medskip
    We now turn to the proof of Theorem~\ref{thm:weighted_H2}.
    \begin{proof}
    \textbf{Step 1: Existence of a weak solution.}
    
	       Let $q_0$ be defined by \eqref{def:q} in terms of the initial data $(\u_0,\e_0)$, so that $q_0$ depends only on $(\u_0,\e_0)\in\X_{\bbeta_f,\bbeta_s}$.
        
    Assume that $(\u^*, p^*, \e^*)$ is a solution to the following system:
	\begin{align}
		\text{PDEs}&\begin{cases}
			-\nabla\cdot (2\mathbb{D}(\mathbf{u}^*) - p^* \mathbb{I})=\f^* &\text{in}\,\,\Omega_f\times (0,T],\\
			\nabla\cdot \mathbf{u^*}=0 &\text{in}\,\,\Omega_f\times (0,T],\\
			-\nabla\cdot (2\mu_1\mathbb{D}(\e^* ) +\mu_2(\nabla\cdot \e ^*) \mathbb{I})=\h^*&\text{in} \,\,
			\Omega_s\times (0,T],
		\end{cases}\label{eq:pde_*}\\
		\text{B.C.}&\begin{cases}
			\u^*=\u_{in}&\text{on}\,\,\Gamma_f\times (0,T],\\
			\partial_t\bm{\eta}^*=\mathbf{u}^*&\text{on}\,\,\Gamma\times (0,T],\\
			(2\mathbb{D}(\mathbf{u}^*) - p^* \mathbb{I})\mathbf{n}=(2\mu_1\mathbb{D}(\e ^*) +\mu_2(\nabla\cdot \e^* ) \mathbb{I})\n &\text{on}\,\,\Gamma\times (0,T],\\
			(2\mu_1\mathbb{D}(\e ^*) +\mu_2(\nabla\cdot \e^* ) \mathbb{I})\n =\mathbf{0}&\text{on}\,\,\Gamma_s\times (0,T],
		\end{cases}\label{eq:pde_*_BC}\\
		\text{I.C.} &\e|_{t=0}=\e_0,\label{eq:pde_*_IC}
	\end{align}
	where \begin{equation}\label{eq:f_h_*}
		\begin{cases}
			&\f^*=-\nabla\cdot (2\mathbb{D}(\mathbf{u}_0) - q_0 \mathbb{I})\in \L^2(\Omega_f),\\
			&\h^*=-\nabla\cdot (2\mu_1\mathbb{D}(\e_0 ) +\mu_2(\nabla\cdot \e_0 ))\in \L^2(\Omega_s).
		\end{cases}
	\end{equation}
	
	This system can be decoupled by viewing it as two sub-problems that are coupled through their interface conditions:
    \begin{itemize}
        \item [1.] The Stokes sub-problem in $\Omega_f$:
        \begin{equation}
            \begin{cases}
			-\nabla\cdot (2\mathbb{D}(\mathbf{u}^*) - p^* \mathbb{I})=\f^* &\text{in}\,\,\Omega_f\times (0,T],\\
			\nabla\cdot \mathbf{u}^*=0 &\text{in}\,\,\Omega_f\times (0,T],\\
            \u^*=\u_{in}&\text{on}\,\,\Gamma_f\times (0,T],\\
            (2\mathbb{D}(\mathbf{u}^*) - p ^*\mathbb{I})\mathbf{n}=(2\mu_1\mathbb{D}(\e^* ) +\mu_2(\nabla\cdot \e^* ) \mathbb{I})\n &\text{on}\,\,\Gamma\times (0,T],
		\end{cases}
        \end{equation}
        \item [2.] The Elasticity sub-problem in $\Omega_s$:
        \begin{equation}
            \begin{cases}
			-\nabla\cdot (2\mu_1\mathbb{D}(\e^* ) +\mu_s(\nabla\cdot \e ^*) \mathbb{I})=\h^*&\text{in} \,\,
			\Omega_s\times (0,T],\\
            \e^*=\int_0^t\u ^*(s)\d s+\e_0&\text{on}\,\,\Gamma\times (0,T],\\
            (2\mu_1\mathbb{D}(\e ^*) +\mu_2(\nabla\cdot \e^* ) \mathbb{I})\n =\mathbf{0}&\text{on}\,\,\Gamma_s\times (0,T].
		\end{cases}.
        \end{equation}
    \end{itemize}
	Applying the solution operators defined in Section \ref{sec:solution_operators}, we can express the solution to these sub-problems formally as:
\begin{align}\label{eq:u_in}
		\begin{cases}
			&(\u^*,p^*)=N_\Gamma \e^* +D_{\Gamma_f}\u_{in}+F_1\f ^*;\\
			&\e^*=D_{\Gamma}(\int_{0}^{t}\u^* (s) \d s+\e_0)+F_2\h^*.
		\end{cases}
	\end{align}
Substituting the expression for $\bm{\eta}^*$ from the second equation into the first, we obtain a closed equation for the fluid velocity $\mathbf{u}^*$:
	\begin{align*}\label{eq:int_uin}(\u^*,p^*)&=N_{\Gamma}D_\Gamma\int_{0}^{t}\u^* (s) \d s+N_\Gamma D_{\Gamma}\e_0+D_{\Gamma_f}\u_{in}+N_\Gamma F_2\h^*+F_1\f ^*,\\
\u^*&=(N_{\Gamma})_1D_\Gamma\int_{0}^{t}\u^* (s) \d s+(N_\Gamma)_1 D_{\Gamma}\e_0+(D_{\Gamma_f})_1\u_{in}+(N_{\Gamma})_1 F_2\h^*+(F_1)_1\f ^*.
	\end{align*}
    Here, the subscript $(\cdot)_1$ denotes the first component of the solution pair returned by the operators.
        
	Since, by Lemma \ref{lem:solution_operator},  the operators satisfy  $N_{\Gamma}\in \mathcal{L}(\H^{2,2}_{\bbeta_s}(\Omega_s);\H^{2,2}_{\bbeta_f}(\Omega_f)\times H^{1,1}_{\bbeta_f}(\Omega_f))$ and $D_\Gamma\in \mathcal{L}(\H^{2,2}_{\bbeta_f}(\Omega_f);\H^{2,2}_{\bbeta_s}(\Omega_s))$, it follows that their composition $(N_\Gamma)_1D_{\Gamma}$ is a bounded linear operator on $\H^{2,2}_{\bbeta_f}(\Omega_f)$. 

    Consequently, by the standard theory of semigroups \cite[Theorem 1.2, Corollary 1.4]{pazy2012semigroups}, the bounded linear operator $(N_\Gamma)_1 D_{\Gamma}$ generates a uniformly continuous (and hence analytic) semigroup ${T(t)}$ on $\mathbf{H}^{2,2}_{\bm{\beta}_f}(\Omega_f)$. This semigroup is given explicitly by the exponential series:
\begin{equation}
T(t):=e^{t(N_\Gamma)_1D_{\Gamma}}=\sum\limits_{k=0}^{\infty }\frac{(t(N_\Gamma)_1D_{\Gamma})^k}{k!}\in \mathcal{L}(\H^{2,2}_{\bbeta_f}(\Omega_f)),
\end{equation}
and satisfies the norm bound
\begin{equation}\label{eq:norm_nound_T}
\|T(t)\|_{\mathcal{L}(\mathbf{H}^{2,2}_{\bm{\beta}_f}(\Omega_f))} \leq e^{C t}
\end{equation}
for some constant $C > 0$ and all $t \geq 0$.

	To utilize this structure, we define an auxiliary variable $\mathbf{v}^*(t) = \int_{0}^{t} \mathbf{u}^*(s) ds$. This definition transforms the integral equation for $\mathbf{u}^*$ into an evolution equation for $\mathbf{v}^*$: 
	\begin{align}\label{eq:v*}
		\partial_t\vv^*&=(N_\Gamma)_1D_{\Gamma}\vv^*(t)+(N_\Gamma)_1D_{\Gamma}\e_0+(D_{\Gamma_f})_1\u_{in}+(N_\Gamma)_1 F_2\h^*+(F_1)_1\f ^*\notag\\
		&=(N_\Gamma)_1D_{\Gamma}\vv^*(t)+\w^*(t),
	\end{align}
	where the inhomogeneous source term is defined as $ \w^*(t):=(N_\Gamma)_1D_{\Gamma}\e_0+(D_{\Gamma_f})_1\u_{in}+(N_\Gamma)_1 F_2\h^*+(F_1)_1\f ^*$. By Lemma \ref{lem:solution_operator}, we have $\mathbf{w}^*(t) \in \mathbf{H}^{2,2}_{\bm{\beta}_f}(\Omega_f)$ for all $t \in [0,T]$, and it satisfies the estimate:
    \begin{align}
        \label{eq:estimate_w*}
        \|\w(t)^*\|_{\mathbf{H}^{2,2}_{\bm{\beta}_f}(\Omega_f)}
        &\leq C \big(\|\e_0\|_{\H^{2,2}_{\bbeta_s}(\Omega_s)}+\|\u_{in}\|_{L^\infty(0,T;\H^{3/2,3/2}_{\bbeta_f}(\Gamma_f))}+\|-\nabla\cdot (2\mathbb{D}(\mathbf{u_0}) - q_0 \mathbb{I})\|_{\L^2_{\bbeta_f}(\Omega_f)}\notag\\
        &\quad +\|-\nabla\cdot (2\mu_1\mathbb{D}(\e_0 ) +\mu_2(\nabla\cdot \e_0 ) \mathbb{I})\|_{\L^2_{\bbeta_s}(\Omega_s)}\big)\notag\\
        &\leq C \big(\|\e_0\|_{\H^{2,2}_{\bbeta_s}(\Omega_s)}+\|\u_{in}\|_{L^\infty(0,T;\H^{3/2,3/2}_{\bbeta_f}(\Gamma_f))}+\|-\nabla\cdot (2\mathbb{D}(\mathbf{u_0}) - q_0 \mathbb{I})\|_{\L^2(\Omega_f)}\notag\\
        &\quad +\|-\nabla\cdot (2\mu_1\mathbb{D}(\e_0 ) +\mu_2(\nabla\cdot \e_0 ) \mathbb{I})\|_{\L^2(\Omega_s)}\big)\notag\\
        &\leq C\big(\|\u_{in}\|_{L^\infty(0,T;\H^{3/2,3/2}_{\bbeta_f}(\Gamma_f))}+\|(\u_0,\e_0)\|_{\X_{\bbeta_f,\bbeta_s}}\big)
    \end{align}
     where we have used the embeddings $\L^2(\Omega_f)\subset \L^2_{\bbeta_f}(\Omega_f)$ and $\L^2(\Omega_s)\subset \L^2_{\bbeta_s}(\Omega_s)$.
    
    The unique solution to the evolution equation \eqref{eq:v*} is therefore given by the variation of constants formula:
	\begin{align*}
		\vv^*(t)&=\int_{0}^{t}T(t-s)\w^*(s)\d s\quad \text{with}\quad \vv^*\in C(0,T;\H^{2,2}_{\bbeta_f}(\Omega_f))
	\end{align*}
    Applying the semigroup bound \eqref{eq:norm_nound_T} and estimate \eqref{eq:estimate_w*} yields the following:
    \begin{align}\label{eq:estimate_v*}
        \|\vv^*(t)\|_{\mathbf{H}^{2,2}_{\bm{\beta}_f}(\Omega_1)}&\leq C\int_0^te^{C(t-s)} \big(\|\u_{in}\|_{L^\infty(0,T;\H^{3/2,3/2}_{\bbeta_f}(\Gamma_f))}+\|(\u_0,\e_0)\|_{\X_{\bbeta_f,\bbeta_s}}\big)\d s\notag\\
        &\leq Ce^{CT} \big(\|\u_{in}\|_{L^\infty(0,T;\H^{3/2,3/2}_{\bbeta_f}(\Gamma_f))}+\|(\u_0,\e_0)\|_{\X_{\bbeta_f,\bbeta_s}}\big).
    \end{align}
    Thus, the solution to the original system \eqref{eq:pde_*}--\eqref{eq:pde_*_IC} or equivalently, \eqref{eq:u_in}, is recovered via:
	\begin{equation}\label{eq:solution*}
		\begin{cases}
			&\e^*=D_{\Gamma}(\vv^* +\e_0)+F_2\h^*,\\
			&(\u^*,p^*)=N_\Gamma \e^* +D_{\Gamma_f}\u_{in}+F_1\f^*.
		\end{cases}
	\end{equation}
	Lemma \ref{lem:solution_operator} and the estimate \eqref{eq:estimate_v*} imply that $(\u^*(t),p^*(t),\e^*(t))\in \mathbf{H}^{2,2}_{\bm{\beta}_f}(\Omega_f)\times H^{1,1}_{\bm{\beta}_f}(\Omega_f)\times\mathbf{H}^{2,2}_{\bm{\beta}_s}(\Omega_s) $ with the bound:
    \begin{align}\label{eq:estimate_u*}
        &\|\u^*(t)\|_{\mathbf{H}^{2,2}_{\bm{\beta}_f}(\Omega_f)}+\|p^*(t)\|_{H^{1,1}_{\bm{\beta}_f}(\Omega_f)}+\|\e^*(t)\|_{\mathbf{H}^{2,2}_{\bm{\beta}_s}(\Omega_s)}\notag\\
        &\leq Ce^{CT}\big(\|\u_{in}\|_{L^\infty(0,T;\H^{3/2,3/2}_{\bbeta_f}(\Gamma_f))}+\|(\u_0,\e_0)\|_{\mathrm{X}_{\bbeta_f,\bbeta_s}}\big).
    \end{align}
    Given the additional regularity assumptions 
    $$
    \partial_t \u_{in}\in L^\infty(0,T;\H^{1/2}(\Gamma_f)) \quad\mbox{and}\quad \partial_{tt}\u_{in}\in L^2(0,T;\H^{1/2,1/2}_{\bbeta_f,00}(\Gamma_{f};\Omega_f)^\prime ) ,
    $$ 
    we can formally differentiate the system \eqref{eq:solution*} with respect to time. This yields the following system for the time derivatives:
	\begin{equation}\label{eq:u_t_eta_tt}
		\begin{cases}
			&\partial_t \e^*=D_{\Gamma} \u^*,\\
			&\partial_{t} \u^*=(N_\Gamma)_1 \partial_t\e^* +(D_{\Gamma_f})_1\partial_t\u_{in},\\
			&\partial_{tt} \e^*=D_{\Gamma}\partial_{t} \u^*,\\
			&\partial_{tt} \u^*=(N_\Gamma)_1 \partial_{tt}\e^* +(D_{\Gamma_f})_1\partial_{tt}\u_{in},\\
			&\partial_{ttt} \e^*=D_{\Gamma}\partial_{tt} \u^*.\\
		\end{cases}
	\end{equation}
    
    By Lemma \ref{lem:solution_operator} and the estimate for $\u ^*$ from \eqref{eq:estimate_u*}, we obtain $\partial_t\e^*(t) \in \mathbf{H}^{2,2}_{\bm{\beta}_s}(\Omega_s)$ for all $t \in [0,T]$ with
    \begin{align}\label{eq:estimate_e*_t}
        &\|\partial_t\e ^*(t)\|_{\mathbf{H}^{2,2}_{\bm{\beta}_s}(\Omega_s)}
        \leq Ce^{CT} \big(\|\u_{in}\|_{L^\infty(0,T;\H^{3/2,3/2}_{\bbeta_f}(\Gamma_f))}+\|(\u_0,\e_0)\|_{\X_{\bbeta_f,\bbeta_s}}\big).
    \end{align}
    
    Since $\partial_t \u_{in}\in L^\infty(0,T;\H^{1/2}(\Gamma_f))$ and $D_{\Gamma_f}\in \mathcal{L}(\H^{1/2}(\Gamma_f);\H^1(\Omega_f)\times L^2(\Omega_f))$ by \eqref{bounded_map_H^1}, it follows that $(D_{\Gamma_f})_1\partial_t\u_{in}(t)\in \H^1(\Omega_f)$ for all $t\in [0,T]$ with 
    \begin{equation}
        \|(D_{\Gamma_f})_1\partial_t\u_{in}(t)\|_{\H^1(\Omega_f)}\leq C \|\partial_t \u_{in}\|_{L^\infty(0,T;\H^{1/2}(\Gamma_f))}.
    \end{equation}
    Combining Lemma \ref{lem:solution_operator} with \eqref{eq:estimate_e*_t}, we deduce that $\partial_t \u^*\in L^\infty(0,T;\H^1(\Omega_f))$ and
    \begin{align}
        \|\partial_t \u^*(t)\|_{\H^1(\Omega_f)}&\leq \|(N_\Gamma)_1 \partial_t\e^*(t)\|_{\mathbf{H}^{2,2}_{\bm{\beta}_f}(\Omega_f)}+\|(D_{\Gamma_f})_1\partial_t\u_{in}(t)\|_{\H^1(\Omega_f)}\notag\\
        &\leq C(\|\partial_t\e ^*(t)\|_{\mathbf{H}^{2,2}_{\bm{\beta}_s}(\Omega_s)}+\|(D_{\Gamma_f})_1\partial_t\u_{in}(t)\|_{\H^1(\Omega_f)})\notag\\
        &\leq Ce^{CT} \big(\|\u_{in}\|_{L^\infty(0,T;\H^{3/2,3/2}_{\bbeta_f}(\Gamma_f))}+\|\partial_t \u_{in}\|_{L^\infty(0,T;\H^{1/2}(\Gamma_f))}+\|(\u_0,\e_0)\|_{\X_{\bbeta_f,\bbeta_s}}\big).\label{eq:estimate_u*t_H1}
    \end{align}
    
    Since $D_\Gamma \in  \mathcal{L}(\H^{1/2}(\Gamma);\H^1(\Omega_s))$ by \eqref{bounded_map_H^1}, we obtain that 
    \begin{align}
        \|\partial_{tt} \e^*(t)\|_{\H^1(\Omega_s)}&\leq C \|\partial_t \u^*(t)\|_{\H^1(\Omega_f)}\notag\\
        &\leq Ce^{CT} \big(\|\u_{in}\|_{L^\infty(0,T;\H^{3/2,3/2}_{\bbeta_f}(\Gamma_f))}+\|\partial_t \u_{in}\|_{L^\infty(0,T;\H^{1/2}(\Gamma_f))}+\|(\u_0,\e_0)\|_{\X_{\bbeta_f,\bbeta_s}}\big).\label{eq:estimate_eta*_tt}
    \end{align}
Since $\nabla\cdot \si_s(\partial_{tt} \e^*)=-\partial_{tt}\h^*=\bm{0}$ in $\Omega_s$, we have  $\si_s(\partial_{tt} \e^*(t)) \in {\mathbb{H}(\mathrm{div},\Omega_s)}$. Then, by Lemma \ref{lem:gamma_n}, the normal trace 
$\si_s(\partial_{tt} \e^*(t)) \n\in \H^{1/2}_{00}(\Gamma;\Omega_s)^\prime  $ satisfies
\begin{equation}
    \|\si_s(\partial_{tt} \e^*(t))\n\|_{\H_{00}^{1/2}(\Gamma;\Omega_s)^\prime }\leq C \|\si_s(\partial_{tt} \e^*(t))\|_{{\mathbb{H}(\mathrm{div},\Omega_s)}}\leq C \|\partial_{tt} \e^*(t)\|_{\H^1(\Omega_s)}.
\end{equation}
Applying the operator $ \tilde{N}_{\Gamma}\in \mathcal{L}(\H_{00}^{1/2}(\Gamma;\Omega_f)^\prime ;\H^1(\Omega_f)\times L^2(\Omega_f))$ by \eqref{bounded_map_H^1}, we obtain  $(N_\Gamma)_1 \partial_{tt}\e^*(t)=(\tilde{N}_{\Gamma})_1(\si_s(\partial_{tt} \e^*(t))\n)\in \H^1(\Omega_f)$ with 
    \begin{align}
        &\|(N_\Gamma)_1 \partial_{tt}\e^*(t)\|_{\H^1(\Omega_f)} \notag\\
        &\leq C \|\si_s(\partial_{tt} \e^*(t))\cdot \n\|_{\H_{00}^{1/2}(\Gamma;\Omega_s)^\prime }\notag \\
        &\leq Ce^{CT} \big(\|\u_{in}\|_{L^\infty(0,T;\H^{3/2,3/2}_{\bbeta_f}(\Gamma_f))}+\|\partial_t \u_{in}\|_{L^\infty(0,T;\H^{1/2}(\Gamma_f))}+\|(\u_0,\e_0)\|_{\X_{\bbeta_f,\bbeta_s}}\big),\label{eq:estimate_e*_tt}
    \end{align}
    where we have used
    \begin{equation}
        \H_{00}^{1/2}(\Gamma;\Omega_f)^\prime =\H_{00}^{1/2}(\Gamma;\Omega_s)^\prime.
    \end{equation}
    
Using $(D_{\Gamma_f})_1\in \mathcal{L}(\H^{1/2,1/2}_{\bbeta_f,00}(\Gamma_f;\Omega_f)^\prime ;\L^2(\Omega_f))$ from Lemma \ref{lem:transition_method} and $$\partial_{tt}\u_{in}\in L^2(0,T;\H^{1/2,1/2}_{\bbeta_f,00}(\Gamma_{f};\Omega_f)^\prime ) ,$$ we derive the estimate:
    \begin{align}
        \|\partial_{tt} \u^*\|_{L^2(0,T;\L^2(\Omega_f))}^2
        &\leq 2\int_0^T \|(N_\Gamma)_1 \partial_{tt}\e^*(t)\|_{\L^2(\Omega_f)}^2+\|(D_{\Gamma_f})_1\partial_{tt}\u_{in}(t)\|_{\L^2(\Omega_f)}^2\d t\notag\\
        &\leq C\int_0^T \|(N_\Gamma)_1 \partial_{tt}\e^*(t)\|_{\H^1(\Omega_f)}^2+\|\partial_{tt}\u_{in}(t)\|_{\H^{1/2,1/2}_{\bbeta_f,00}(\Gamma_{f};\Omega_f)^\prime }^2\d t\notag\\
        &\leq Ce^{CT} \big(\|\u_{in}\|^2_{L^\infty(0,T;\H^{3/2,3/2}_{\bbeta_f}(\Gamma_f))}+\|\partial_t \u_{in}\|^2_{L^\infty(0,T;\H^{1/2}(\Gamma_f))}\notag\\
        &\quad +\|\partial_{tt} \u_{in}\|^2_{L^2(0,T;\H^{1/2,1/2}_{\bbeta_f,00}(\Gamma_{f};\Omega_f)^\prime )}+\|(\u_0,\e_0)\|^2_{\X_{\bbeta_f,\bbeta_s}}\big)\quad  (\text{using \eqref{eq:estimate_e*_tt}}).\label{eq:estimate_u*_tt}
    \end{align}
Using the relation in  \eqref{eq:u_t_eta_tt}, we  write $\partial_{ttt} \e^*=D_{\Gamma}\partial_{tt}\u^*=D_{\Gamma}(N_\Gamma)_1 \partial_{tt}\e^*+D_{\Gamma}(D_{\Gamma_f})_1\partial_{tt}\u_{in}$. Then 
\begin{align}
    \|\partial_{ttt} \e^*\|_{L^2(0,T;\L^2(\Omega_s))}^2
        &\leq 2\int_0^T \|D_{\Gamma}(N_\Gamma)_1 \partial_{tt}\e^*(t)\|_{\L^2(\Omega_s)}^2+\|D_{\Gamma}(D_{\Gamma_f})_1\partial_{tt}\u_{in}(t)\|_{\L^2(\Omega_s)}^2\d t\notag\\
        &= \int_0^T \|D_{\Gamma}(N_\Gamma)_1 \partial_{tt}\e^*(t)\|_{\L^2(\Omega_s)}^2+\|D_{\Gamma}\gamma_\Gamma[(D_{\Gamma_f})_1\partial_{tt}\u_{in}(t)]\|_{\L^2(\Omega_s)}^2\d t\quad \notag\\
        &\hspace{120pt}(\text{using Lemma \ref{lem:trace_L2}})\notag\\
        &\leq C \int_0^T \|(N_\Gamma)_1 \partial_{tt}\e^*(t)\|_{\H^1(\Omega_s)}^2+\|\partial_{tt}\u_{in}(t)\|_{\H^{1/2,1/2}_{\bbeta_f,00}(\Gamma_{f};\Omega_f)^\prime}^2\d t\quad \notag\\
        &\hspace{120pt}(\text{using \eqref{bounded_map_H^1},  Lemma \ref{lem:transition_method} and Lemma  \ref{lem:trace_L2}})\notag\\
        &\leq Ce^{CT} \big(\|\u_{in}\|^2_{L^\infty(0,T;\H^{3/2,3/2}_{\bbeta_f}(\Gamma_f))}+\|\partial_t \u_{in}\|^2_{L^\infty(0,T;\H^{1/2}(\Gamma_f))}\notag\\
        &\quad +\|\partial_{tt} \u_{in}\|^2_{L^2(0,T;\H^{1/2,1/2}_{\bbeta_f,00}(\Gamma_{f};\Omega_f)^\prime )}+\|(\u_0,\e_0)\|^2_{\X_{\bbeta_f,\bbeta_s}}\big)\quad  (\text{using \eqref{eq:estimate_e*_tt}}).\label{eq:estimate_eta*_ttt}
\end{align}

    From \eqref{eq:estimate_u*t_H1}--\eqref{eq:estimate_eta*_ttt}, we conclude that
    \begin{equation}
        \partial_t\u^*\in H^1(0,T;\L^2(\Omega_f)),\quad \partial_{tt}\e^*\in H^1(0,T;\L^2(\Omega_s)).
    \end{equation}
    with the estimate
    \begin{align}
        &\|\partial_t \u^*\|_{H^1(0,T;\L^2(\Omega_f))}+\|\partial_{tt}\e^*\|_{H^1(0,T;\L^2(\Omega_s))} \notag\\
        &\leq Ce^{CT} \big(\|\u_{in}\|_{L^\infty(0,T;\H^{3/2,3/2}_{\bbeta_f}(\Gamma_f))}+\|\partial_t \u_{in}\|_{L^\infty(0,T;\H^{1/2}(\Gamma_f))}\notag\\
        &\quad +\|\partial_{tt} \u_{in}\|_{L^2(0,T;\H^{1/2,1/2}_{\bbeta_f,00}(\Gamma_{f};\Omega_f)^\prime )}+\|(\u_0,\e_0)\|_{\X_{\bbeta_f,\bbeta_s}}\big).\label{eq:estimate_H^1_ut_etatt}
    \end{align}
   
    Finally, using the uniqueness of the solution and the initial compatibility conditions \eqref{eq:initial_compatibility}, we verify consistency at $t=0$:
    \begin{equation}
    (\u^*(0),p^*(0))=N_\Gamma \e_0 +D_{\Gamma_f}\u_{in}(0)+F_1\f^*(0)=(\u_0,q_0)
    \end{equation}
    
	Let $(\tilde{\u},\tilde{p},\tilde{\e})=(\u-\u^*,p-p^*,\e-\e^*)$ define the difference between the solution of the original problem \eqref{eq:stoke_elasticity_t}-\eqref{con:initial}  and the one constructed in the previous step. Subtracting the systems  satisfied by $ (\u,p,\e)$ and  $(\u^*,p^*,\e^*)$, we find that $(\tilde{\mathbf{u}}, \tilde{p}, \tilde{\bm{\eta}})$ satisfies the following homogeneous system for $t \in (0, T]$:
		\begin{align}\label{eq:new_dirichlet_0}
		\text{PDEs}&\begin{cases}
			\partial_t \tilde{\u}-\nabla\cdot (2\mathbb{D}(\tilde{\u}) -\tilde{p} \mathbb{I})=\tilde{\f} &\text{in}\,\,\Omega_f\times (0,T],\\
			\nabla\cdot \tilde{\u}=0 &\text{in}\,\,\Omega_f\times (0,T],\\
			-\nabla\cdot (2\mu_1\mathbb{D}(\tilde{\e} ) +\mu_2(\nabla\cdot \tilde{\e} ) \mathbb{I})=\tilde{\h}&\text{in} \,\,
			\Omega_s\times (0,T],
		\end{cases}\\
		\text{B.C.}&\begin{cases}
			\tilde{\u}=\mathbf{0}&\text{on}\,\,\Gamma_f\times (0,T],\\
			\partial_t\tilde{\e}=\tilde{\u}&\text{on}\,\,\Gamma\times (0,T],\\
			(2\mathbb{D}(\tilde{\u}) -\tilde{p} \mathbb{I})\mathbf{n}=(2\mu_1\mathbb{D}(\tilde{\e} ) +\mu_2(\nabla\cdot \tilde{\e} ) \mathbb{I})\n &\text{on}\,\,\Gamma\times (0,T],\\
			(2\mu_1\mathbb{D}(\tilde{\e} ) +\mu_2(\nabla\cdot \tilde{\e} ) \mathbb{I})\n =\mathbf{0}&\text{on}\,\,\Gamma_s\times (0,T],\label{eq:new_BC}
		\end{cases}\\
		\text{I.C.} &(\tilde{\u } ,\tilde{\e }  ,\partial_t\tilde{\e } )|_{t=0}=(\tilde{\u}_0,\tilde{\e}_0,\tilde{\e}_1).\label{eq:new_initial}
	\end{align}
	The new forcing terms and initial data are given by:
	\begin{equation}
		\begin{cases}
			\tilde{\f}&=\f-\f^*-\partial_t\u^*\in H^1(0,T;\L^2(\Omega_f)),\\
			\tilde{\h}&=\h-\h^*-\partial_{tt}\e^*\in H^1(0,T;\L^2(\Omega_s)),\\
			\tilde{\u}_0&=\u _0-\u^*(0)=\mathbf{0},\\
			\tilde{\e}_0&=\e _0-\e^*(0)=\mathbf{0},\\
			\tilde{\e}_1&=\e _1-\partial_t\e^*(0)\in \H^1(\Omega_s).\\
		\end{cases}
	\end{equation}
	Crucially, this reduced system enjoys homogeneous initial conditions $(\tilde{\u}_0,\tilde{\e}_0)=(0,0)$, a homogeneous Dirichlet boundary condition on $\Gamma_f$, and a homogeneous Neumann boundary condition on $\Gamma_s$. Existence and uniqueness of a solution $(\tilde{\mathbf{u}}, \tilde{p}, \tilde{\bm{\eta}})$ follow by arguments analogous to those in \cite[Theorem 2.3]{du2004semidiscrete}. Applying the same  method yields the estimate
		\begin{align}
			&\|\tilde{\u }(t)\|_{\L^2(\Omega_f)}+\|\tilde{\u }\|_{L^2(0,T;\H^1(\Omega_f)}+\|\tilde{\e }(t)\|_{\H^1(\Omega_s)}+\|\tilde{p}\|_{L^2(0,T;L^2(\Omega_f))}\notag\\
			& +\|\partial_t\tilde{\u }(t)\|_{\L^2(\Omega_f)}+\|\partial_t\tilde{\u }\|_{L^2(0,T;\H^1(\Omega_f)}+\|\partial_t\tilde{\e }(t)\|_{\H^1(\Omega_s)}+\|\partial_{tt}\tilde{\e }(t)\|_{\L^2(\Omega_s)}\notag\\
            &\leq Ce^{CT}\big (\|\tilde{\f}\|_{H^1(0,T;\L^2(\Omega_f))}+\|\tilde{\h}\|_{H^1(0,T;\L^2(\Omega_s))}+\|\tilde{\e}_1\|_{\H^1(\Omega_s)}\big )\notag \\
			&\leq Ce^{CT}\big(\|\f\|_{H^1(0,T;\L^2(\Omega_f))}+\|\h\|_{H^1(0,T;\L^2(\Omega_s))}+\|\u_{in}\|_{L^\infty(0,T;\H^{3/2,3/2}_{\bbeta_f}(\Gamma_f))}+\|\partial_t \u_{in}\|_{L^\infty(0,T;\H^{1/2}(\Gamma_f))}\notag\\
        &\quad +\|\partial_{tt} \u_{in}\|_{L^2(0,T;\H^{1/2,1/2}_{\bbeta_f,00}(\Gamma_{f};\Omega_f)^\prime )}+\|\e_1\|_{\H^1(\Omega_s)}+\|(\u_0,\e_0)\|_{\X_{\bbeta_f,\bbeta_s}}\big)\quad \text{(\eqref{eq:estimate_e*_t} and \eqref{eq:estimate_H^1_ut_etatt} are used )}\notag\\
        &\leq Ce^{CT}S,
		\end{align}
        where \begin{align*}
S&=\|\f\|_{L^\infty(0,T;\L^2(\Omega_f))}+\|\partial_t\f\|_{L^2(0,T;\L^2(\Omega_f))}+\|\h\|_{H^1(0,T;\L^2(\Omega_s))}+\|\u_{in}\|_{L^\infty(0,T;\H^{3/2,3/2}_{\bbeta_f}(\Gamma_f))}\notag\\
        &\quad +\|\partial_t \u_{in}\|_{L^\infty(0,T;\H^{1/2}(\Gamma_f))}+\|\partial_{tt} \u_{in}\|_{L^2(0,T;\H^{1/2,1/2}_{\bbeta_f,00}(\Gamma_{f};\Omega_f)^\prime )}+\|\e_1\|_{\H^1(\Omega_s)}+\|(\u_0,\e_0)\|_{\X_{\bbeta_f,\bbeta_s}}.
        \end{align*}
        Combining this with the estimates \eqref{eq:estimate_u*}, \eqref{eq:estimate_e*_t}, \eqref{eq:estimate_u*t_H1}, \eqref{eq:estimate_eta*_tt}  for $(\mathbf{u}^*, p^*, \bm{\eta}^*)$, we conclude that 
        \begin{align}
			&\|\u (t)\|_{\L^2(\Omega_f)}+\|\u \|_{L^2(0,T;\H^1(\Omega_f)}+\|\e(t)\|_{\H^1(\Omega_s)}+\|p\|_{L^2(0,T;L^2(\Omega_f))}\notag\\
			& +\|\partial_t\u(t)\|_{\L^2(\Omega_f)}+\|\partial_t\u\|_{L^2(0,T;\H^1(\Omega_f)}+\|\partial_t\e(t)\|_{\H^1(\Omega_s)}+\|\partial_{tt}\e(t)\|_{\L^2(\Omega_s)}\notag\\
        &\leq Ce^{CT}S.
		\end{align}
        
       \medskip
\textbf{Step 2: Derivation of $\H^{2,2}_{\bbeta}$-regularity estimates.}

    From the elasticity equation in \eqref{eq:stoke_elasticity_t} at $t=0$,
    \begin{equation}
        \h (0)=\partial_{tt}\e (0)-\nabla\cdot (2\mu_1\mathbb{D}(\e_0 ) +\mu_2(\nabla\cdot \e_0 ) \mathbb{I}).
    \end{equation}
Applying the triangle and Cauchy–Schwarz inequalities yields:
    \begin{align}
        \|\h(t)\|_{\L^2(\Omega_s)}=&\|\h(0)+\int_0^t\partial_t\h(s)\d s\|_{\L^2(\Omega_s)}\notag\\
        \leq&  \|\h(0)\|_{\L^2(\Omega
        _s)}+t^{1/2}\|\partial_t\h\|_{L^2(0,T;\L^2(\Omega
        _s)}\notag\\
        \leq&\|\partial_{tt}\e (0)\|_{\L^2(\Omega_s)}+\|\nabla\cdot (2\mu_1\mathbb{D}(\e_0 ) +\mu_2(\nabla\cdot \e_0 )\|_{\L^2(\Omega_s)}+t^{1/2}\|\partial_t\h\|_{L^2(0,T;\L^2(\Omega
        _s)}\notag\\
        \leq& Ce^{CT}S.
    \end{align}
  Consequently,
    \begin{equation}
        \f-\partial_t \u \in L^\infty(0,T;\L^2(\Omega
        _f)),\quad \h-\partial_{tt} \e \in L^\infty(0,T;\L^2(\Omega
        _s)),
    \end{equation}
    and the following bound holds:
    \begin{align}\label{eq:estimates_f_u}
        &\|\f-\partial_t\u  \|_{L^\infty(0,T;\L^2(\Omega
        _f))}+\|\h-\partial_{tt} \e\|_{L^\infty(0,T;\L^2(\Omega
        _s))}\leq Ce^{CT}S.
    \end{align}
	
    We begin by reformulating the original problem \eqref{eq:stoke_elasticity_t}--\eqref{con:initial}. The solution $(\mathbf{u}, p, \bm{\eta})$ satisfies the system:
	\begin{align}
		\text{PDEs}&\begin{cases}
			-\nabla\cdot (2\mathbb{D}(\mathbf{u}) - p \mathbb{I})=\f-\partial_t\u  &\text{in}\,\,\Omega_f\times (0,T],\\
			\nabla\cdot \mathbf{u}=0 &\text{in}\,\,\Omega_f\times (0,T],\\
			-\nabla\cdot (2\mu_1\mathbb{D}(\e ) +\mu_2(\nabla\cdot \e ) \mathbb{I})=\h-\partial_{tt}\e &\text{in} \,\,
			\Omega_s\times (0,T],
		\end{cases}\label{eq:pde_t_tt}\\
		\text{B.C.}&\begin{cases}
			\u=\u_{in}&\text{on}\,\,\Gamma_f\times (0,T],\\
			\partial_t\bm{\eta}=\mathbf{u}&\text{on}\,\,\Gamma\times (0,T],\\
			(2\mathbb{D}(\mathbf{u}) - p \mathbb{I})\mathbf{n}=(2\mu_1\mathbb{D}(\e ) +\mu_2(\nabla\cdot \e ) \mathbb{I})\n &\text{on}\,\,\Gamma\times (0,T],\\
			(2\mu_1\mathbb{D}(\e ) +\mu_2(\nabla\cdot \e ) \mathbb{I})\n=\mathbf{0}&\text{on}\,\,\Gamma_s\times (0,T],
		\end{cases}\\
		\text{I.C.} &({\u } ,{\e }  ,\partial_t{\e } )|_{t=0}=({\u}_0,{\e}_0,{\e}_1).
	\end{align}
   This system can be decoupled into two subproblems that are coupled through the interface conditions:
	\begin{align}
		\text{(Fluid Subproblem)} &\begin{cases}
			-\nabla\cdot (2\mathbb{D}(\mathbf{u}) - p \mathbb{I})=\f-\partial_t\u  &\text{in}\,\,\Omega_f\times (0,T],\\
			\nabla\cdot \mathbf{u}=0 &\text{in}\,\,\Omega_f\times (0,T],\\
			\u=\u_{in}&\text{on}\,\,\Gamma_f\times (0,T],\\
			(2\mathbb{D}(\mathbf{u}) - p \mathbb{I})\mathbf{n}=(2\mu_1\mathbb{D}(\e ) +\mu_2(\nabla\cdot \e ) \mathbb{I})\n &\text{on}\,\,\Gamma\times (0,T],
		\end{cases}\\
		\text{(Elasticity Subproblem)}&\begin{cases}
			-\nabla\cdot (2\mu_1\mathbb{D}(\e ) +\mu_2(\nabla\cdot \e ) \mathbb{I})=\h-\partial_{tt}\e &\text{in} \,\,
			\Omega_s\times (0,T],\\
			\e=\int_{0}^{t}\u(s)\d s+\e_0 &\text{on}\,\,\Gamma\times (0,T],\\
			(2\mu_1\mathbb{D}(\e ) +\ur (\nabla\cdot \e ) \mathbb{I})\n=\mathbf{0}&\text{on}\,\,\Gamma_s\times (0,T].
		\end{cases}
	\end{align}
    Applying the solution operators defined in Section \ref{sec:solution_operators}, we decouple the FSI system into two interconnected subproblems:
	\begin{align}
		\begin{cases}
			&(\u,p)=N_\Gamma \e +D_{\Gamma_f}\u_{in}+F_1(\f-\partial_t\u);\\
			&\e=D_{\Gamma}(\int_{0}^{t}\u (s) \d s+\e_0)+F_2(\h-\partial_{tt}\e).
		\end{cases}
	\end{align}
    Substituting the expression for $\bm{\eta}$ from the second equation into the first, we obtain a closed equation for the fluid velocity $\mathbf{u}$:
	\begin{align}
    (\u,p)&=N_{\Gamma}D_\Gamma\int_{0}^{t}\u (s) \d s+N_\Gamma D_{\Gamma}\e_0+D_{\Gamma_f}\u_{in}+N_\Gamma F_2(\h-\partial_{tt}\e)+F_1(\f-\partial_t\u),\notag\\
\u&=(N_{\Gamma})_1D_\Gamma\int_{0}^{t}\u (s) \d s+(N_\Gamma)_1 D_{\Gamma}\e_0+(D_{\Gamma_f})_1\u_{in}+(N_{\Gamma})_1 F_2(\h-\partial_{tt}\e)+(F_1)_1(\f-\partial_t\u).
	\end{align}
To analyze this equation, we define the auxiliary variable $\mathbf{v}(t) = \int_{0}^{t} \mathbf{u}(s) ds$. This transformation converts the integral equation for $\u $  into an evolution equation for $\mathbf{v}$:
	\begin{align}\label{eq:v_FSI}
		\partial_t\vv&=(N_\Gamma)_1D_{\Gamma}\vv(t)+\w^\#(t),
	\end{align}
	where $ \w^\#(t)=(N_\Gamma)_1D_{\Gamma}\e_0+(D_{\Gamma_f})_1\u_{in}+(N_\Gamma)_1 F_2(\h-\partial_{tt}\e)+(F_1)_1(\f-\partial_t\u)$.  By  Lemma \ref{lem:solution_operator} and \eqref{eq:estimates_f_u},  we have $\mathbf{w}^\#(t) \in \mathbf{H}^{2,2}_{\bm{\beta}_f}(\Omega_f)$ for every $t \in [0,T]$ and satisfies the estimate:
    \begin{align}\label{eq:estimate_w**}
        &\|\w^\#(t)\|_{\mathbf{H}^{2,2}_{\bm{\beta}_f}(\Omega_f)}\leq Ce^{CT}S.
    \end{align} 
    The unique solution to the evolution equation \eqref{eq:v_FSI} is therefore given by the variation of constants formula:
	\begin{align*}
		\vv(t)&=\int_{0}^{t}T(t-s)\w^\#(s)\d s\quad \text{with}\quad \vv\in C(0,T;\H^{2,2}_{\bbeta_f}(\Omega_f)),
	\end{align*}
    where ${T(t)}$ is the semigroup generated by $(N_\Gamma)_1 D_{\Gamma}$ which is bounded operator on $\H^{2,2}_{\bbeta_f}(\Omega_f)$.
    Applying the semigroup bound \eqref{eq:norm_nound_T} and the estimate \eqref{eq:estimate_w**} yields:
    \begin{align}\label{eq:estimate_v**}
        \|\vv(t)\|_{\mathbf{H}^{2,2}_{\bm{\beta}_f}(\Omega_f)}
        &\leq C\int_0^te^{C(t-s)} e^{CT}S 
        \leq Ce^{CT}S.
    \end{align}
    Thus, the solution to the original system \eqref{eq:stoke_elasticity_t}--\eqref{con:initial} is recovered by:
	\begin{equation}\label{eq:solution**}
		\begin{cases}
			&\e=D_{\Gamma}(\vv +\e_0)+F_2(\h-\partial_{tt}\e),\\
			&(\u,p)=N_\Gamma \e +D_{\Gamma_f}\u_{in}+F_1(\f -\partial_t\u ).
		\end{cases}
	\end{equation}
    Combining Lemma~\ref{lem:solution_operator} with the estimates \eqref{eq:estimate_v**} and \eqref{eq:estimates_f_u}, we obtain the higher regularity
    \begin{equation}
        (\u(t),p(t),\e(t))\in \mathbf{H}^{2,2}_{\bm{\beta}_f}(\Omega_f)\times H^{1,1}_{\bm{\beta}_f}(\Omega_f)\times\mathbf{H}^{2,2}_{\bm{\beta}_s}(\Omega_s), 
    \end{equation}
     with the quantitative bound:
    \begin{align}\label{eq:estimate_u**}
        &\|\u(t)\|_{\mathbf{H}^{2,2}_{\bm{\beta}_f}(\Omega_f)}+\|p(t)\|_{\mathbf{H}^{1,1}_{\bm{\beta}_f}(\Omega_f)}+\|\e(t)\|_{\mathbf{H}^{2,2}_{\bm{\beta}_s}(\Omega_s)}\leq Ce^{CT}S.
    \end{align}
    Using \eqref{eq:pde_t_tt}
and \eqref{eq:estimates_f_u}, we further obtain
    \begin{align}
        &\quad \|-\nabla\cdot (2\mathbb{D}(\mathbf{u}(t)) - p (t)\mathbb{I})\|_{\L^2(\Omega_f)}+\|-\nabla\cdot (2\mu_1\mathbb{D}(\e (t)) +\mu_2(\nabla\cdot \e(t) ) \mathbb{I})\|_{\L^2(\Omega_s)}\notag\\
        &=\|\f(t)-\partial_t\u(t) \|_{\L^2(\Omega_f)}+\|\h(t)-\partial_{tt}\e(t) \|_{\L^2(\Omega_s)}\notag\\
        &\leq Ce^{CT}S.\label{eq:estimate_u_p_L2}
    \end{align}
    Recall that $q=p-r$, where $r$  defined via \eqref{map:definition_M} and satisfies the estimate
    \begin{align}\label{esti_r}
        \|\nabla r(t)\|_{\L^2(\Omega_f)}
        &\leq C \|-\nabla\cdot (2\mathbb{D}(\mathbf{u}(t)) - p (t)\mathbb{I})\|_{\L^2(\Omega_f)}.
    \end{align}
    Consequently, 
    \begin{align}\label{estimate:Du-q_L^2}
        \|-\nabla\cdot (2\mathbb{D}(\mathbf{u}(t)) - q (t)\mathbb{I})\|_{\L^2(\Omega_f)}
        &\leq \|-\nabla\cdot (2\mathbb{D}(\mathbf{u}(t)) - p (t)\mathbb{I})\|_{\L^2(\Omega_f)}+\|\nabla r(t)\|_{\L^2(\Omega_f)}\notag\\
        &\leq Ce^{CT}S.
    \end{align}
    Finally, combining \eqref{eq:estimate_u**}--\eqref{estimate:Du-q_L^2} with the definition of the norm on $\X_{\bbeta_f,\bbeta_s}$, we obtain
    \begin{align}
        &\|(\u(t),\e(t))\|_{\X_{\bbeta_f,\bbeta_s}}\leq Ce^{CT}S.
    \end{align}
 \end{proof}

\section{Conclusion}\label{sec:conclusion}

In this work, we have established existence, uniqueness, and $\H^{2,2}_{\bbeta}$-regularity of solutions for a fluid--structure interaction (FSI) problem posed on polygonal domains with non-homogeneous Dirichlet boundary conditions. We addressed the regularity challenges arising from geometric singularities at corners where the interface meets the boundary by employing tailored weighted Sobolev spaces, $\mathbf{H}^{k,l}_{\boldsymbol{\beta}}(\Omega)$. 

Our analysis is constructive: we developed specialized solution operators that decouple the coupled FSI system into manageable subsystems, allowing precise local estimates and a rigorous characterization of singular behavior. The theoretical results presented here provide a foundation for designing optimal mesh refinement strategies to achieve improved approximation accuracy, thereby enhancing computational performance. Overall, this work offers a robust framework with potential applications in scientific computing, numerical analysis, and the simulation of FSI phenomena in non-smooth domains.


\section*{Acknowledgement}
This work was supported by the National Natural Science Foundation of China (Grant No. 12525111) and the Research Grants Council of Hong Kong (Grant No. PolyU/RFS2324-5S03).

\bibliographystyle{abbrv} 
\bibliography{sample}

\end{document}